\documentclass{article}

\usepackage[centertags]{amsmath}
\usepackage{amssymb}
\usepackage{mathptmx}
\usepackage{mathtools}
\usepackage[colorlinks=true,linkcolor=blue,citecolor=blue]{hyperref}
\usepackage{subfig}
\usepackage{amsthm}
\usepackage{adjustbox}
\usepackage{csquotes}
\usepackage{longtable}

\usepackage{tikz}
\usetikzlibrary{decorations.pathmorphing}
\usepackage{graphicx, xcolor, soul}

\usepackage[colorinlistoftodos,prependcaption,textsize=tiny]{todonotes}

\usepackage{url}
\usepackage[T1]{fontenc}

\usepackage{tikz}
\usetikzlibrary{calc}

\title{Survey: Sixty Years of Douglas--Rachford}

\author{Scott B. Lindstrom\\CARMA\\University of Newcastle \and Brailey Sims\\CARMA\\University of Newcastle}

\date{\today}

\def\Id{\hbox{\rm Id}}

\newcommand{\argmin}{\ensuremath{\operatorname{argmin}}}
\newcommand{\Fix}{\ensuremath{\operatorname{Fix}}}
\newcommand{\A}{\ensuremath{\mathbb A}}
\newcommand{\B}{\ensuremath{\mathbb B}}

\newtheorem{theorem}{Theorem}[section]

\newtheorem{proposition}[theorem]{Proposition}
\theoremstyle{definition}
\newtheorem{definition}[theorem]{Definition}

\newtheorem{remark}[theorem]{Remark}

\begin{document}
	
	\maketitle
	
	\subsubsection*{Dedication}

This work is dedicated to the memory of Jonathan M. Borwein our greatly missed friend, mentor, and colleague. His influence on both the topic at hand, as well as his impact on the present authors personally, cannot be overstated.

\begin{abstract}
	The Douglas--Rachford method is a splitting method frequently employed for finding zeroes of sums of maximally monotone operators. When the operators in question are normal cones operators, the iterated process may be used to solve \emph{feasibility} problems of the form: $\text{Find}\quad x \in \bigcap_{k=1}^N S_k.$ The success of the method in the context of closed, convex, nonempty sets $S_1,\dots,S_N$ is well-known and understood from a theoretical standpoint. However, its performance in the nonconvex context is less understood yet surprisingly impressive. This was particularly compelling to Jonathan M. Borwein who, intrigued by Elser, Rankenburg, and Thibault's success in applying the method for solving Sudoku Puzzles 
	began an investigation of his own. We survey the current body of literature on the subject, and we summarize its history. We especially commemorate Professor Borwein's celebrated contributions to the area.
\end{abstract}

\section{Introduction}\label{sec:intro}
In 1996  Heinz Bauschke and Jonathan Borwein broadly classified the commonly applied projection algorithms for solving convex feasibility problems as falling into four categories. These were: best approximation theory, discrete models for image reconstruction, continuous models for image reconstruction, and subgradient algorithms \cite{BB96}. One such celebrated iterative process has been known by many names in many contexts and is possibly best known as the \emph{Douglas--Rachford} method (DR). 

DR is frequently used for the more general problem of finding a zero of the sum of maximally monotone operators, which itself is a generalization of the problem of minimizing a sum of convex functions. Many volumes could be written on monotone operator theory, convex optimization, and splitting algorithms specifically, the definitive work being that of Bauschke and Combettes \cite{BC}; the story of DR is inextricably entwined with each of these.

More recently, the method has become famous for its surprising success in solving nonconvex feasibility problems, notwithstanding the lack of theoretical justification. The recent investigation of these methods in the nonconvex setting has been both motivated by and advanced through experimental application of the algorithms to nonconvex problems in a variety of different settings. In many cases impressive performance has been observed despite having previously been thought of as ill-adapted to projection algorithms.

The task of choosing what to include in a condensed survey of DR is thus necessarily difficult. We therefore choose to adopt an approach which balances reasonable brevity with the goal that a reader unfamiliar with DR should be able to at least glean the following: the basic history of the method, an understanding of the various motivating contexts in which it has been ``discovered,'' an appreciation for the diversity of problems to which it is applied, and a sense of which research topics are currently being explored.

\subsection{Outline}
This paper is divided into four sections: \vspace{0.15cm}

\begin{longtable}{l p{9.5cm}}
	\textbf{Section \ref{sec:intro}} & In \ref{preliminaries}, we provide preliminaries on Douglas--Rachford and feasibility. In \ref{History}, we briefly motivate its history and explain how feasibility problems are a special case of finding a zero for a sum of maximal monotone operators, and in \ref{operator_sums} we explore its use for finding zeros of maximal monotone operator sums, including its connection with ADMM in \ref{ADMM}. In \ref{subsec:Nsets}, we analyse the ways in which it has been extended from $2$ set feasibility problems to $N$ set feasibility problems.\\
	& \\
	\textbf{Section \ref{convex_setting}} & We consider the role of DR in solving convex feasibility problems. In \ref{convex_convergence} we catalogue some of the convergence results, and in \ref{convex_applications} we mention some of its better known applications.\\
	& \\
	\textbf{Section \ref{nonconvex}} & We consider the context of nonconvex feasibility. We first consider discrete problems in \ref{nonconvex_discrete} and go on to discuss hypersurface problems in \ref{nonconvex_hypersurfaces}. In \ref{transversal_section}, we explore some of the possibly nonconvex convergence results which employ notions of regularity and transversality. In \ref{nonconvex_minimization} we describe some of the recent work applying DR for nonconvex minimization problems.\\
	& \\
	\textbf{Section \ref{summary}} & Finally we mention two open problems and summarize the current state of research in the field.\\
	\textbf{Appendix \ref{appendix:ADMM}} & This appendix provides a more detailed summary of Gabay's exposition on the connection between DR and ADMM.
\end{longtable}

\subsection{Preliminaries}\label{preliminaries}

The method of alternating projections (AP) and the Douglas--Rachford method (DR) are frequently used to find a \emph{feasible point} (point in the intersection) of two closed constraint sets $A$ and $B$ in a Hilbert space $H$. The \emph{feasibility problem} is
\begin{equation}\label{eqn:feasibility_problem}
\text{Find} \quad x \in A \cap B.
\end{equation}

The projection onto a subset $C$ of $H$ is defined for all $x \in H$ by $$\mathbb{P}_C(x) := \left \{ z \in C : \|x - z\| = \inf_{z' \in C}\|x - z'\|\right \}.$$

Note that $\mathbb{P}_C$ is, generically, a set-valued map where values may be empty or contain more than one point. In the cases of interest to us $\mathbb{P}_C$ has nonempty values (indeed throughout $\mathbb{P}_C$ is nonempty and so $C$ is said to be \emph{proximal}), and in order to simplify both notation and implementation, we will work with a selector for $\mathbb{P}_C$, that is a map $P_C:H \rightarrow C: x \mapsto P_C(x) \in \mathbb{P}_C(x)$, so $P_C^2=P_C$.

When $C$ is nonempty, closed, and convex the projection operator $P_C$ is uniquely determined by the variational inequality
$$
\left(x - P_C(x), c - P_C(x)\right) \leq 0, \quad \textrm{for all\ } c \in C,
$$
and is a \emph{firmly nonexpansive}mapping; that is for all $ x,y\in H$
\begin{equation*}
\|P_C x - P_C y \|^2 + \|(I-P_C)x-(I-P_C)y\| \leq \|x-y\|^2.
\end{equation*}
See, for example, \cite[Chapter 4]{BC}. When $C$ is a closed subspace it is also a self-adjoint linear operator \cite[Corollary 3.22]{BC}.

The reflection mapping through the set $C$ is defined by
$$R_C := 2P_C - I,$$
where $I$ is the identity map.

\begin{definition}[\emph{Method of Alternating Projections}]
	For two closed sets $A$ and $B$ and an initial point $x_0 \in H$, the method of alternating projections (AP) generates a sequence $(x_n)_{n=1}^\infty$ as follows:
	\begin{equation}\label{def:AP}
	x_{n+1} := P_{B}P_{A}x_n.
	\end{equation}	
\end{definition}

\begin{definition}[\emph{Douglas--Rachford Method}]\label{def:DRMethod}
	For two closed sets $A$ and $B$ and an initial point $x_0 \in H$, the Douglas--Rachford method (DR) generates a sequence $(x_n)_{n=1}^\infty$ as follows:
	\begin{equation}\label{def:DR}
	x_{n+1} \in T_{A,B}(x_n) \quad \text{where} \quad T_{A,B} := \frac{1}{2}\left( I + R_{B}R_{A}\right).
	\end{equation}	
\end{definition}
DR is often referred to as \emph{reflect-reflect-average}. Both DR and AP are special cases of averaged relaxed projection methods. We denote a \emph{relaxed projection} by	
\begin{equation}\label{eq:defRS}
R_C^\gamma(x):=(2-\gamma)(P_C -\Id)+\Id,
\end{equation}
for a fixed \emph{reflection parameter} $\gamma \in [0,2)$. Observe that when $\gamma = 0$, the operator $R_C^{\gamma=0} = 2 P_C -\Id$ is the standard \emph{reflection} employed by DR, and for $\gamma=1$ we obtain the \emph{projection}, $R_C^\gamma = R_C^1 =  P_C$. For $\gamma\in (1,2)$ the operator $R_C^\gamma$ can be called an \emph{under-relaxed projection} following \cite{DePierro}. Here we are using the terminology in \eqref{eq:defRS}. However, the reader is cautioned that in some articles, $R_C^\gamma$ is written as $P_C^\gamma$, while in others the role of $\gamma$ is reversed so that $\gamma=2$ corresponds to a reflection and $\gamma=0$ is the identity: $\gamma(P_C-\Id)+\Id$.

In addition to using  relaxed projections as in  \eqref{eq:defRS}, the averaging step of the Douglas--Rachford iteration can also be relaxed by choosing an arbitrary point on the interval between the second reflection and the initial iterate. This can be parametrised by some $\lambda \in (0,1]$. Accordingly we define a $\lambda$-averaged relaxed sequence $\{x_n\}$ by,
\begin{equation}\label{DRsequence}
x_n := \left(T_{A^\gamma,B^\mu}^\lambda \right)^n x_0:= \left(\lambda(R_B^\mu \circ R_A^\gamma)+(1-\lambda)\Id \right)^n x_0.
\end{equation}
When $\lambda = \gamma = \mu = 1$, this is the sequence generated by alternating projections \eqref{def:AP}, and for $\lambda = 1/2$ and $\gamma = \mu = 0$, this is the standard Douglas--Rachford sequence \eqref{def:DR}. For $\gamma = \mu = 0$ and $\lambda=1$, this is the \emph{Peaceman-Rachford} sequence \cite{peaceman1955numerical} (see also Lions \& Mercier \cite[Algorithm 1]{LM}).

We note that the framework introduced here does not cover all possible projection methods. For example, one may want to vary the parameters $\gamma$, $\mu$ and $\lambda$ on every step, or consider other variations of Douglas--Rachford operators (see \cite{FranNewMethod} for example). Single steps of the AP and DR methods are illustrated in
Figure~\ref{fig:APandDR}, which originally appeared in \cite{DLR18}.

\begin{figure}[ht]
	\subfloat[One step of alternating projections]{
		\begin{tikzpicture}[scale=2.6]
		\fill[cyan,fill opacity=0.4] (1-.25,1-.1) to (2.0-.25,0-.1)
		to [out=180,in=270] (.465-.25,.415-.1)
		to [out=90,in=205] cycle;
		\node at (1.1-.25,.2-.1) {$A$};
		
		
		\fill[red,fill opacity=0.4] (1.42-.2,.3+.2) to (1.2-.2,.3+.2) to (1.3-.2,.5+.2) to (1.5-.2,.9+.2) to [out=60,in=135] (1.9-.2,.9+.2) to [out=315,in=0] (1.5-.2,.3+.2) to cycle;
		\node [] at (1.6,.9) {$B$}; 					
		
		
		
		\draw [gray, dashed] (.625,1.25) -- (0,0);
		\node [above right] at (.625,1.25) {$H_A$};
		
		\draw [fill,black] (-.25,.75) circle [radius=0.015];	
		\node [left] at (-.25,.75) {$x$};
		\draw [black,->] (-.2,.725) -- (.2,.525);
		\draw [fill,black] (.25,.5) circle [radius=0.015];
		\node [left] at (.25,.45) {$P_Ax$};

		\draw [black,->] (.325,.5) -- (.925,.5);
		\draw [fill,black] (1,.5) circle [radius=0.015];
		\node [below left] at (1,.5) {$P_B P_A x$};
		
		\draw [gray, dashed] (1,1.25) -- (1,0);
		\node [right] at (1,1.25) {$H_B$};
		
		\end{tikzpicture}}
	\subfloat[One step of Douglas--Rachford method]{
		\begin{tikzpicture}[scale=2.6]
		\fill[cyan,fill opacity=0.4] (1-.25,1-.1) to (2.0-.25,0-.1)
		to [out=180,in=270] (.465-.25,.415-.1)
		to [out=90,in=205] cycle;
		\node at (1.1-.25,.2-.1) {$A$};
		
		
		\fill[red,fill opacity=0.4] (1.42-.2,.3+.2) to (1.2-.2,.3+.2) to (1.3-.2,.5+.2) to (1.5-.2,.9+.2) to [out=60,in=135] (1.9-.2,.9+.2) to [out=315,in=0] (1.5-.2,.3+.2) to cycle;
		\node [] at (1.65,1) {$B$}; 					
		
		
		
		\draw [gray, dashed] (.625,1.25) -- (0,0);
		\node [above right] at (.625,1.25) {$H_A$};
		
		\draw [fill,black] (-.25,.75) circle [radius=0.015];	
		\node [left] at (-.25,.75) {$x$};
		\draw [black,->] (-.2,.725) -- (.2,.525);
		\draw [fill,black] (.25,.5) circle [radius=0.015];
		\node [left] at (.25,.45) {$P_Ax$};
		\draw [blue,->] (.3,.475) -- (.7,.275);
		\draw [fill,blue] (.75,.25) circle [radius=0.015];	
		\node [left, blue] at (.75,.2) {$R_A x$};
		
		\draw [blue,->] (1.05,.55) -- (1.2,.7);
		\draw [fill,blue] (1.25,.75) circle [radius=0.015];
		\node [above,blue] at (1.25,.75) {$R_B R_A x$};
		\draw [black,->] (.8,.3) -- (.95,.45);
		\draw [fill,black] (1,.5) circle [radius=0.015];
		\draw [blue,->] (1.05,.55) -- (1.2,.7);
		\node [right] at (1,.5) {$P_B R_A x$};
		
		\draw [gray, dashed] (.25,1.25) -- (1.6,-.1);
		\node [above left] at (.25,1.25) {$H_B$};
		
		\draw [purple, <->] (-.2,.75) -- (1.2,.75);
		\draw [fill,purple] (.5,.75) circle [radius=0.015];
		\node [below,purple] at (.5,.75) {$T_{A,B}x$};
		
		\end{tikzpicture}\label{fig:APandDRright}}
	\caption{The operator $T_{A,B}$.}\label{fig:APandDR}
\end{figure}
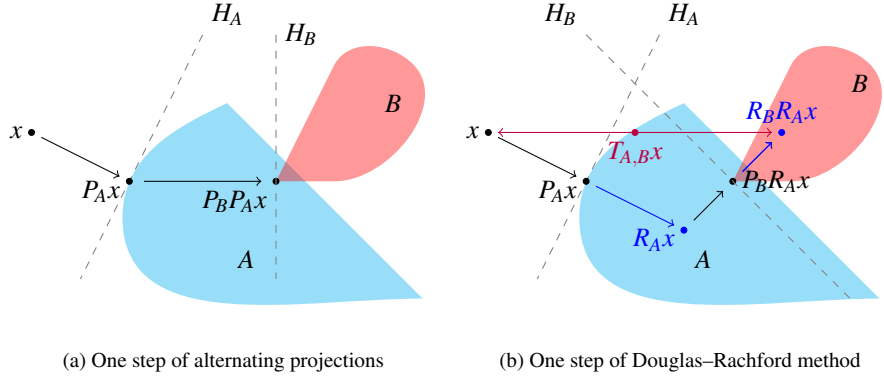

\begin{definition}
	The fixed point set for a mapping $T:H\rightarrow H$ is $\Fix T = \{x \in H \;|\; Tx = x \}$ (in the case when $T$ is set-valued $\Fix T = \{x \in H \;|\; x \in Tx \}$.
\end{definition}

\subsection{History}\label{History}

Projection methods date at least as far back as 1933 when J. von Neumann considered the method of alternating projections when $A$ and $B$ are affine subsets of $H$ establishing its norm convergence to $P_{A\cap B}(x_0)$ \cite{JVN}. In 1965 Bregman showed that in the more general setting where $A$ and $B$ are closed convex sets AP converges weakly to a point in $A\cap B$ \cite{bregman1965method}(see also \cite{BB96}). In 2002 Hundal \cite{Hundal} provided an example in infinite dimensions of a hyperplane and closed cone for which AP fails to converge in norm. However the cone constructed by Hundal is somewhat unnatural. In \cite{BST2015} Borwein, Sims, and Tam  explored the possibility of norm convergence for sets occurring more naturally in applications.

Sixty years ago the Douglas--Rachford method was introduced, somewhat indirectly, in connection with nonlinear heat flow problems \cite{DR}; see \cite{Moursi} for a thorough treatment of the connection with the form we recognize today. The definitive statement of the weak convergence result was given by Lions and Mercier in the more general setting of maximal monotone operators \cite{LM}. We will first state the problem and result, and then explain the connection. The problem is
\begin{equation}\label{eqn:operator_sum_problem}
\text{Find}\quad x \quad \text{such that}\quad 0 \in (\A+\B)x.
\end{equation}
Let the resolvent for a set-valued mapping $F$ be defined by $J_F^\lambda := (\Id+\lambda F)^{-1}$ with $\lambda>0$. The classical result is as follows.
\begin{theorem}[Lions \& Mercier \cite{LM}]\label{thm:LionsandMercier} Assume that
	$\A,\B$ are maximal monotone operators with $\A + \B$ also maximal monotone, then for
	\begin{equation}\label{LM}
	T_{\A,\B}:X\rightarrow X :x\mapsto J_\B^\lambda(2J_\A^\lambda-I)x+(I-J_\A^\lambda)x
	\end{equation}
	the sequence given by $x_{n+1}=T_{\A,\B}x_n$ converges weakly to some $v \in H$ as $n\rightarrow \infty$ such that $J_\A^\lambda v$ is a zero of $\A + \B$.
\end{theorem}
The \emph{normal cone} to a set $C$ at $x\in C$ is $N_C(x) = \{y\in H : (y,c - x)\leq 0\ \textrm{for all\ } c\in C\}$. The normal cone operator associated to $C$ is
\begin{equation}\label{normalconeoperator}
N_C:H\rightarrow H:x \mapsto \left\{\begin{array}{lr} N_C(x),& \textrm{when\ } x\in C \\ \emptyset, &\textrm{when\ } x\notin C. \end{array}\right.
\end{equation}
See, for example, \cite[Definition 6.37]{BC}. 
One may think of the feasibility problem \eqref{eqn:feasibility_problem} as a special case of the optimization problem
\begin{equation}\label{eqn:minimization_problem}
\text{Find}\quad x \in \argmin \left\{ \iota_A + \iota_B \right\}
\end{equation}
where the indicator function $\iota_C$ for a set $C$ is defined by
\begin{equation}
\iota_C : H \rightarrow \mathbb{R}^\infty \quad \text{by} \quad \iota_C: x\mapsto \begin{cases}
0 & \text{if} \; x \in C\\
\infty & \text{otherwise}
\end{cases}.
\end{equation}
Whenever $A$ and $B$ are closed and convex, $\iota_A$ and $\iota_B$ are lower semicontinuous and convex, and their subdifferential operators $\partial \iota_A = N_A$ and $\partial \iota_B = N_B$ are maximal monotone. In this case, under satisfactory constraint qualifications on $A,B$ to guarantee the sum rule for subdifferentials $\partial (\iota_A +\iota_B) = \partial \iota_A + \partial \iota_B$ (see \cite[Corollary 16.38]{BC}), the problem \eqref{eqn:minimization_problem} reduces to
\begin{equation}
\text{Find} \quad x \quad \text{such that} \quad 0 \in \left(\partial \iota_A + \partial \iota_B\right)(x)= \left(N_A + N_B \right)(x)
\end{equation}
which we recognize as \eqref{eqn:operator_sum_problem}. Seen through this lens, two set convex feasibility is a special case of an extremely common problem in convex optimization: that of minimizing a sum of two convex functions $f+g$ where $\mathbb{A}= \partial f$ and $\mathbb{B}= \partial g$.
This illuminates its close relationship to many other proximal iteration methods, including the various augmented Lagrangian techniques with which it is often studied in tandem (see subsection~\ref{ADMM}).

Where $\A = N_A$ and $\B = N_B$ are the normal cone operators for closed convex sets $A$ and $B$, then the resolvents $J_\A^\lambda,J_\B^\lambda$ are the projection operators $P_A,P_B$ respectively, $T_{\A,\B} = \frac{1}{2}R_B R_A + \frac{1}{2}\Id$ is what we recognize as the operator of the usual Douglas--Rachford method\footnote{An operator $T:D\rightarrow H$ with $D \neq \emptyset$ satisfies $T=J_A$ where $A:=T^{-1}-\Id$. Moreover, $T$ is  firmly nonexpansive if and only if $A$ is monotone, and $T$ is firmly nonexpansive with full domain if and only if $A$ is maximally monotone. See \cite[Proposition 23.7]{BC} for details.}, and $J_\A^\lambda v = P_A v \in A \cap B$ is a solution for the feasibility problem \eqref{eqn:feasibility_problem}.
For details, see, for example, \cite[Example 23.4]{BC}.

Rockafellar \cite{rockafellar1974conjugate} and Brezis \cite{Brezis1971} (as cited in \cite{attouch1979maximality}) showed that the condition ${\rm dom}\A \cap {\rm int}{\rm dom}\B \ne \emptyset$ is sufficient to ensure that $\A$ and $\B$ maximal monotone implies that $\A + \B$ is also maximal monotone. In 1979, Hedy Attouch showed that the weaker condition $0 \in {\rm int}({\rm dom}\A - {\rm dom}\B)$ is sufficient \cite{attouch1979maximality}.

However, Attouch's condition may not be satisfied if $\A = N_A$ and $\B=N_B$ where $A$ and $B$ meet at a single point, since ${\rm dom}N_A = A$ and ${\rm dom}N_B = B$. In the following theorem, Bauschke, Combettes, and Luke \cite{BCL} showed that in the case of the feasibility problem \eqref{eqn:feasibility_problem} the requirement $\A + \B$ be maximal monotone may be relaxed.
\begin{theorem}[{\cite[Fact 5.9]{BCL}}]\label{thm:BCL}
	Suppose $A,B \subseteq H$ are closed and convex with non-empty intersection. Given $x_0 \in H$ the sequence of iterates defined by $x_{n+1}:=T_{A,B}x_n$ converges weakly to an $x \in Fix T_{A,B}$ with $P_{A}x \in A \cap B$.
\end{theorem}
It should be noted that Zarantonello gave an example showing that when $C$ is not closed and affine $P_C$ need not be weakly continuous \cite{zarantonello1971projections} (see also \cite[ex. 4.12]{BC}). Despite the potential discontinuity of the resolvent $J_A^\lambda$, Svaiter later demonstrated that $J_A^\lambda x_n$ converges weakly to some $v \in {\rm zer} (\A+\B)$ \cite{Svaiter}.

Theorem \ref{thm:BCL} relies on the firm nonexpansivity of $T_{A,B}$. This is an immediate consequence of the fact that it is a $1/2$-average of $R_BR_A$ with the identity and that $P_A$, $P_B$ are themselves firmly nonexpansive so that $R_A$, $R_B$ and hence $R_BR_A$ are nonexpansive. The proof of theorem \ref{thm:LionsandMercier} similarly relies on the firm nonexpansivity of $J_\A^\lambda$ and $J_\B^\lambda$; its requirement that $\A+\B$ be maximal monotone was later relaxed by Svaiter \cite{Svaiter}.

\subsection{Through the Lens of Monotone Operator Sums}\label{operator_sums}

While our principle interest lies in the less general setting of projection operators, much of the investigation of the Douglas--Rachford algorithm has centered on analysis of the problem \eqref{eqn:operator_sum_problem}. We provide a brief summary.

In 1989 (\cite{EB92}), Jonathan Eckstein and Dimitri Bertsekas motivated the advantage of $T_{\B,\A}$ among resolvent methods as a \emph{splitting method}: a method which employs separate computation of resolvents for $\A$ and $\B$ in lieu of attempting to compute the resolvent of $\A + \B$ directly. They showed that, in the case where $\text{zer}(\A+\B)=\emptyset$, the sequence \eqref{def:DR} is unbounded, a useful diagnostic observation. They also demonstrated that with exact evaluation of resolvents the Douglas--Rachford method is a special case of the \emph{proximal point algorithm} \cite[Theorem 6]{EB92} in the sense of iterating a resolvent operator \cite{rockafellar1976monotone}:
\begin{align}\label{iterated_resolvent}
&x_{n+1}:=J_{\delta_n A} \; \text{where}\; \delta_n > 0,\;\sum_{n\in \mathbb{N}}\delta_n = +\infty,\\
\text{and}\;\; &A:H\rightarrow 2^H\;\; \text{is maximally monotone with}\;\; {\rm zer}A \neq \emptyset.
\end{align}
For more information on this characterization, see \cite[Theorem 23.41]{BC}. In his PhD dissertation \cite{eckstein1989splitting}, Eckstein went on to show that the Douglas--Rachford operator may, however, fail to be a \emph{proximal mapping} \cite[Theorem 27.1]{BC} in the sense of satisfying
\begin{align}\label{iterated_prox}
x_{n+1}:=\text{prox}_{\delta_n f} x_n \quad \text{where}&\quad \delta_n > 0,\;\sum_{n\in \mathbb{N}}\delta_n = +\infty,\;\text{and}\;f\in \Gamma_0(H) \\
\text{and}&\quad \text{prox}_{\delta_n f}x:=\underset{y \in X}{\argmin}\left(\delta_n f(y)+ \frac{1}{2}\|x-y\|^2 \right).\nonumber
\end{align}
Since $\text{prox}_{\delta_n f} = J_{\partial (\delta_n f)}$ (see, for example, \cite{BC}), clearly \eqref{iterated_prox} implies \eqref{iterated_resolvent}. This is also why, in the literature, \emph{Douglas--Rachford splitting} is frequently described in terms of prox operators as
\begin{align}
\textbf{Step 0.}&\quad \text{Set initial point}\quad x_0 \; \; \text{and parameter} \;\; \eta > 0\label{LP_notation} \\
\textbf{Step 1.}&\quad \text{Set}\;\; \begin{cases}
y_{n+1} &\in \underset{y}{\argmin}\left \{f(y)+\frac{1}{2\eta}\|y-x_n\|^2 \right \}=\text{prox}_{\eta f}(x_n)\\
z_{n+1} &\in \underset{z}{\argmin}\left \{g(z)+\frac{1}{2\eta}\|2y_{n+1}-x_n-z\|^2 \right \}=\text{prox}_{\eta g}(2y_{n+1}-x_n)\\
x_{n+1} &= x_n + (z_{n+1}-y_{n+1})
\end{cases},\nonumber
\end{align}
which simplifies to \eqref{def:DR} when $f:=\iota_A$ and $g:=\iota_B$ are indicator functions for convex sets. See, for example, \cite{LP,patrinos2014douglas}.

In 2018, Heinz Bauschke, Jason Schaad, and Xianfu Wang \cite{bauschke2018douglas} investigated Douglas--Rachford operators which fail to satisfy \eqref{iterated_prox}, demonstrating that for linear relations which are maximally monotone $T_{\A,\B}$ generically does not satisfy \eqref{iterated_prox}. 

In 2004, Combettes provided an excellent illumination of the connections between the Douglas--Rachford method, the Peaceman-Rachford method, the backward-backward method, and the forward-backward method \cite{comb}. He also established the following result on a perturbed, relaxed extension of DR, which we quote with minor notation changes.
\begin{theorem}[Combettes, 2004]Let $\gamma \in ]0,+\infty[$, let $(\nu_n)_{n \in \mathbb{N}}$ be a sequence in $]0,2[$, and let $(a_n)_{n \in \mathbb{N}}$ and $(b_n)_{n \in \mathbb{N}}$ be sequences in ${H}$. Suppose that $0 \in {\rm ran}(\A+\B)$, $\sum_{n \in \mathbb{N}}\nu_n (2-\nu_n)=+\infty$, and $\sum_{n \in \mathbb{N}}(\|a_n\|+\|b_n\|)<+\infty$. Take $x_0 \in \mathcal{H}$ and set
	\begin{equation*}
	(\forall n \in \mathbb{N})\; x_{n+1}=x_n+\nu_n\left(J_{\gamma \A} \left(2(J_{\gamma \B} x_n + b_n)-x_n \right)+a_n - \left(J_{\gamma \B x_n}+b_n\right) \right).
	\end{equation*}
	Then $(x_n){n \in \mathbb{N}}$ converges weakly to some point $x \in {H}$ and $J_{\gamma \B}x \in (\A+\B)^{-1}(0)$.
\end{theorem}
At the same time Eckstein and Svaiter conducted a similar investigation through the lens of Fej\'{e}r monotonicity, allowing the proximal parameter to vary from operator to operator and iteration to iteration \cite{eckstein2008family}.

In 2011, Bingsheng He and Xiaoming Yuan provided a simple proof of the worst case $O(1/k)$ convergence rate in the case where the maximally monotone operators $\A$ and $\B$ are continuous on $\mathbb{R}^n$ \cite{he2015convergence}.

In 2011 \cite{BBHM2012}, Bauschke, Radu Bo{\c{t}, Warren Hare, and Walaa Moursi analyzed the Attouch-Th\'{e}ra duality of the problem \eqref{eqn:operator_sum_problem}, providing a new characterization of $\Fix T_{\B,\A}$. In their 2013 article \cite{BHM2014} Bauschke, Hare, and Moursi introduced a ``normal problem'' associated with \eqref{eqn:operator_sum_problem} which introduces a perturbation based on an infimal displacement vector (see equation~\eqref{eqn:displacementvector}). In 2014, they went on to rigorously investigate the range of $T_{\A,\B}$ \cite{BHM2016}.
	
	In 2015, Combettes and Pesquet introduced a random sweeping block coordinate variant, along with an analogous variant for the forward-backward method \cite{combettes2015stochastic}. In so doing, they furnished a thorough investigation of quasi-Fej\'er monotonicity.
	
	In 2017 Bauschke, Moursi, and Lukens \cite{BLM2017} provided a detailed unpacking of the connections between the original context of Douglas and Rachford \cite{DR} and the classical statement of the weak convergence provided by Lions and Mercier \cite{LM}. In addition, they provided numerous extensions of the original theory in the case where $\A$ and $\B$ are maximally monotone and affine, including results in the infinite dimensional setting.
	
	In the same year, Pontus Giselsson and Stephen Boyd established bounds for the rates of global linear convergence under assumptions of strong convexity of $g$ (where $\mathbb{B}=\partial g$) and smoothness, with a relaxed averaging parameter \cite{giselsson2017linear}. Giselsson also provided tight global linear convergence rate bounds in the more general setting of monotone inclusions \cite{giselsson2017tight}, namely: when one of $\mathbb{A}$ or $\mathbb{B}$ is strongly monotone and the other cocoercive, when one of $\mathbb{A}$ or $\mathbb{B}$ is both strongly monotone and cocoercive, and when one of $\mathbb{A}$ or $\mathbb{B}$ is strongly monotone and Lipschitz continuous.	In the case where one operator is strongly monotone and Lipschitz continuous, Giselsson demonstrated that the linear convergence rate bounds provided by Lions and Mercier are not tight. In his analysis, he introduced and made use of \emph{negatively averaged operators}---$T$ such that that $-T$ is averaged---proving and exploiting the fact that averaged maps of negatively averaged operators are contractive, in order to obtain the linear convergence results. 
	
	In 2018, Moursi and Lieven Vandenberghe \cite{moursi2018douglas} supplemented Giselsson's work by providing linear convergence results in the case where $\mathbb{A}$ is Lipschitz continuous and $\mathbb{B}$ is strongly monotone, a result result that is not symmetric in $\mathbb{A}$ and $\mathbb{B}$ except when $\mathbb{B}$ is a linear mapping.
	
	The DR operator has also been employed as a step in the construction of a more complicated iterated method. For example, in 2015, Luis Brice{\~n}o-Arias considered the problem of finding a zero for a sum of a normal cone to a closed vector subspace of $H$, a maximally monotone operator, and a cocoercive operator. They provided weak convergence results for a method which employs a DR step applied to the normal cone operator and the maximal monotone operator \cite{briceno2015forward}.
	
	Recently, Minh Dao and Hung Phan \cite{dao2018adaptive} have introduced what they call an adaptive Douglas--Rachford splitting algorithm in the context where one operator is strongly monotone and the other weakly monotone. 
	
	Svaiter has also analysed the semi-inexact and fully inexact cases where, respectively, one or both proximal subproblems are solved only approximately, within a relative error tolerance \cite{svaiter2018weakly}.
	
	The definitive modern treatment of the above history---including the most detailed version of the exposition from \cite{BLM2017} on the connections between the contexts of Douglas and Rachford \cite{DR} and Lions and Mercier \cite{LM}---was given by Walaa Moursi in her PhD dissertation \cite{Moursi}.
	
	\subsubsection{Connection with method of multipliers (ADMM)}\label{ADMM}
	
	We provide here an abbreviated discussion of the connection between Douglas--Rachford method and the so-called \emph{method of multipliers} or \emph{ADMM} (alternating direction method of multipliers). For a more detailed exposition, see Appendix~\ref{appendix:ADMM}.
	
	In 1983 \cite{Gabay}, Daniel Gabay showed that, under appropriate constraint qualifications, the Lagrangian method of Uzawa applied to finding 
	\begin{equation}
	\mathbf{p}:=\underset{v \in V}{\inf} \{F(Bv)+G(v)\},
	\end{equation}	
	where $B$ is a linear operator with adjoint $B^*$ and $F,G$ are convex, is equivalent to DR in the Lions and Mercier sense of iterating resolvents \eqref{LM} applied to the problem of finding
	\begin{equation}
	\mathbf{d}:=\underset{\mu \in H}{\inf} \{G^*(-B^*\mu)+F^*(\mu)\}
	\end{equation}	
	where the former is the primal value and the latter is the dual value associated through Fenchel Duality. See, for example, \cite[Theorem 3.3.5]{BL}. We have presented here a more specific case of his result, namely where $B^t=B^*$; the more general version is in Appendix~\ref{appendix:ADMM}.
	
	Gabay gave to this method what is now the commonly accepted name \emph{method of multipliers}. It is also frequently referred to as \emph{alternating direction method of multipliers (ADMM)}. Gabay went on to also consider an analysis of the Peaceman-Rachford algorithm \cite{peaceman1955numerical} (see also Lions \& Mercier \cite[Algorithm 1]{LM}). Because of this connection, DR, PR and ADMM are frequently studied together. Indeed, another name by which ADMM is known is the \emph{Douglas--Rachford ADM}. 
	
	\begin{remark}[On a point of apparently common confusion]\label{rem:history}
		
		In the literature, we have found it indicated that the close relationship between the ADMM and the iterative schemes in Douglas and Rachford's article \cite{DR} and in Lions and Mercier's article \cite{LM} was explained by Chan and Glowinski in 1978 \cite{chan1978finite}. However, both Glowinski and Marroco's 1975 paper \cite{glowinski1975approximation} and Glowinski and Chan's 1978 paper \cite{chan1978finite} predate Lions and Mercier's 1979 paper \cite{LM}, and neither of them contains any reference to Douglas' and Rachford's article \cite{DR}. 
		
		Lions and Mercier made a note that DR (which they called simply \emph{Algorithm II}) is equivalent to one of the penalty-duality methods studied in 1975 by Gabay and Mercier \cite{gabay1976dual} and by Glowinski and Marocco \cite{glowinski1975approximation}. In both of these articles, the method under consideration is simply identified as \emph{Uzawa's algorithm}. The source of the confusion remains unclear, but the explicit explanation of the connection that we have followed is that of Daniel Gabay in 1983 \cite{Gabay}. In fact, clearly explaining the connection appears to have been one of his main intentions in writing his 1983 book chapter.
	\end{remark}
	
	Reasonable brevity precludes an in-depth discussion of Lagrangian duality beyond establishing the connection of ADMM with Douglas--Rachford. Instead, we refer the interested reader to a recent survey of Moursi and Yuriy Zinchenko \cite{moursi2018note}, who drew Gabay's work to the attention of the present authors. We refer the reader also to the sources mentioned in Remark~\ref{rem:history}, to Glowinski, Osher, and Yin's recent book on splitting methods \cite[Chapter~2]{glowinski2017splitting}, and to the following selected resources, which are by no means comprehensive: \cite{bertsekas2015convex,he20121,eckstein2015understanding,he2015non,glowinski2014alternating,fortin2000augmented,elser2017matrix}.
	
	\subsection{Extensions to $N$ sets}\label{subsec:Nsets}
	
	The method of alternating projections, and the associated convergence results, readily extend to the feasibility problem for $N$ sets
	\begin{equation}\label{eqn:Nsets}
	\textrm{Find}\quad x\in \cap_{k=1}^N S_k,
	\end{equation}
	to yield the method of \emph{cyclic projections} that involves iterating $T_{S_1S_2\cdots S_N} = P_{S_N}P_{S_{N-1}}\cdots P_{S_1}$.
	
	However, even for three sets the matching extension of DR,
	$$
	x_{n+1}\ =\ \frac{1}{2}\left(I\ +\ R_{S_3}R_{S_2}R_{S_1} \right)(x_n)
	$$
	may cycle and so fail to solve the feasibility problem. See Figure~\ref{fig:cyclicfail}, an example due to Sims that has previously appeared in \cite{Tam}.
	\begin{figure}
		\begin{center}
			\begin{tikzpicture}[scale=1.5]
			\draw [<->,blue] (0,2) -- (0,-0.5);
			\node [below right,blue] at (0,-0.5) {$S_1$};
			\draw [<->,purple] (-1.025,1.169) -- (1.525,-.3031);
			\node [above left,purple] at (-1.025,1.169) {$S_3$};
			\draw [<->,red] (1.025,1.169) -- (-1.525,-.3031);
			\node [above right,red] at (1.025,1.169) {$S_2$};	
			
			\draw [fill,black] (-1,0) circle [radius=0.015];
			\node [above left] at (-1,0) {$x = R_{S_3}R_{S_2}R_{S_1}x$};
			
			\draw [black,dashed,->] (-.8,0) -- (.8,0);
			\draw [fill,black] (1,0) circle [radius=0.015];
			\node [above right] at (1,0) {$R_{S_1}x$};
			
			\draw [black,dashed,<-] (.1,1.5588) -- (.9,.1732);
			\draw [fill,black] (0,1.732) circle [radius=0.015];
			\node [above left] at (0,1.732) {$R_{S_2}R_{S_1}x$};
			
			\draw [black,dashed,->] (-.1,1.5588) -- (-.9,.1732);
			
			\end{tikzpicture}
		\end{center}
		\caption{The algorithm $x_n:=(\frac{1}{2}R_C R_B R_A + \frac12 \Id)^n x_0 $ may cycle.}\label{fig:cyclicfail}
	\end{figure}
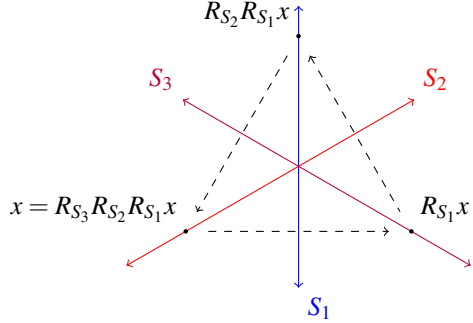
	
	The most commonly used extension of DR from $2$ sets to $N$ sets is Pierra's product space method \cite{Pierra}. More recently Borwein and Tam have introduced a cyclic variant \cite{BTam}.
	
	\subsubsection{Pierra's Product Space Reformulation: ``Divide and Concur'' Method}\label{sec:productspace}
	
	To apply DR for finding $x \in \cap_{k=1}^N S_k \ne \emptyset$, we may work in the Hilbert product space $\mathbf{H} = H^{N}$ as follows.
	\begin{align}
	\text{Let}\quad	S&:= S_1 \times \dots \times S_N\nonumber \\
	\text{and}\quad	D&:= \{(x_1,\dots,x_N)\in \mathbf{H} : x_1=x_2=\dots = x_N \}\label{eqn:B}
	\end{align}
	and apply the DR method to the two sets $S$ and $D$. The product space projections for $x=(x_1,\dots,x_N)\in \mathbf{H}$ are
	\begin{align*}  P_S(x_1,\dots,x_N)&=(P_{S_1}(x_1),\dots, P_{S_N}(x_N)), \\
	\textrm{and} \quad P_D(x_1,\dots,x_N)&=\left(\frac{1}{N}\sum_{k=1}^N x_k,\dots,\frac{1}{N}\sum_{k=1}^N x_k\right).
	\end{align*}
	The method was first nicknamed \emph{divide and concur} by Simon Gravel and Veit Elser \cite{GE}---the latter of whom credits the former for the name \cite{Elseremail}---and the diagonal set $D$ in this context is referred to as the \emph{agreement set}. It is clear that any point $x \in S \cap D$ has the property that $x_1 = x_2 = \dots = x_N \in \cap_{k=1}^N S_k$. It is also clear that $D$ is a closed subspace of $\mathbf{H}$ (so, $P_D$ is weakly continuous) and that, when $S_1,\dots,S_N$ are convex, so too is $S$.
	
	The form of $P_D$ and its consequent linearity allows us to readily unpack the product space formulation to yield the iteration,
	$$
	\left(x_k(n+1)\right)_{k=1}^N\ = \left(x_k(n) - a(n)\ +\ 2A(n) -P_{S_k}\left(x_k(n)\right)\right)_{k=1}^N,
	$$
	where $a(n) = \frac{1}{N}\sum_{k=1}^N x_k(n)$ and $A(n) = \frac{1}{N}\sum_{k=1}^N P_{S_k}\left(x_k(n)\right)$, under which in the convex case the sequence of successive iterates weakly converges (by theorem~\ref{thm:BCL}) to a limit $\left(x_1(\infty),x_2(\infty),\cdots,x_N(\infty)\right)$ for which $P_{S_k}\left(x_k(\infty)\right)$ is, for any $k = 1,2,\cdots,N$, a solution to the $N$-set feasibility problem.
	
	A product space schema can also be applied with AP instead of DR, to yield the method of \emph{averaged projections},
	$$
	x_{n+1} = \frac{1}{N}\sum_{i=1}^N P_i(x_i).
	$$
	
	\subsubsection{Cyclic Variant: Borwein-Tam Method}	The cyclic version of DR, also called the \emph{Borwein-Tam} method, is defined by
	\begin{equation}
	T_{[S_1 S_2 \dots S_N]}:=T_{S_N,S_1}T_{S_{N-1}S_N}\dots T_{S_2,S_3}T_{S_1,S_2},
	\end{equation}
	where each $T_{S_i,S_j}$ is as defined in \eqref{def:DR}. The key convergence result is as follows.
	\begin{theorem}[Borwein \& Tam, 2014]
		Let $S_1,\dots, S_N \subset H$ be closed and convex with nonempty intersection. Let $x_0 \in H$ and set
		\begin{equation}
		x_{n+1}:=T_{[S_1 S_2 \dots S_N]}x_n.
		\end{equation}
		Then $x_n$ converges weakly to $x$ which satisfies $P_{S_1}x = P_{S_2}x = \dots = P_{S_N}x \in \cap_{k=1}^N S_k$.
	\end{theorem}
	For a proof, see \cite[Theorem 3.1]{BTam} or \cite[Theorem 2.4.5]{Tam}, the latter of which---Matthew Tam's dissertation---is the definitive treatise on the cyclic variant.
	
	\subsubsection{Cyclically Anchored Variant (CADRA)}
	
	Bauschke, Noll, and Phan provided linear convergence results for the Borwein-Tam method in the finite dimensional case in the presence of transversality \cite{bauschke2015linear}. At the same time, they introduced the \emph{Cyclically Anchored Douglas--Rachford Algorithm (CADRA)} defined closed, convex sets $A$ (the \emph{anchor} set) and $(B_i)_{i \in \{1,\dots,m\}}$ where
	\begin{align}
	&A\cap \bigcap_{i \in \{1,\dots,m\}}B_i \neq \emptyset \nonumber \\
	\text{and}\quad&(\forall i \in \{1,\dots,m\})T_i = P_{B_i}R_A+\Id -P_A, \quad Z_i = \Fix T_i, \quad Z=\bigcap_{i \in \{1,\dots,m\}}Z_i.\nonumber \\
	\text{where}\quad& (\forall n \in \mathbb{N})\quad x_{n+1}:=Tx_n \quad \text{where}\quad T:=T_m\dots T_2 T_1.\label{CADRA}
	\end{align}
	When $m=1$, CADRA becomes regular DR, which is not the case for the Borwein-Tam method. The convergence result is as follows.
	\begin{theorem}{CADRA (Bauschke, Noll, Phan, 2015 \cite[Theorem 8.5]{bauschke2015linear})}
		The sequence $(x_n)_{n \in \mathbb{N}}$ from \eqref{CADRA} converges weakly to $x \in Z$ with $P_A x \in A\cap \bigcap_{i \in \{1,\dots,m\}}B_i$. Convergence is linear if one of the following hold:
		\begin{enumerate}
			\item X is finite-dimensional and $\text{ri}A\cap \bigcap_{i \in \{1,\dots,m\}}\text{ri}B_i \neq \emptyset$
			\item $A$ and each $B_i$ is a subspace with $A+B_i$ closed and that $(Z_i)_{i \in \{1,\dots,m\}}$ is boundedly linearly regular.
		\end{enumerate}
	\end{theorem}
	
	\subsubsection{String-averaging and block iterative variants}
	
	In 2016, Yair Censor and Rafiq Mansour introduced the string-averaging DR (SA-DR) and block-iterative DR (BI-DR) variants \cite{censor2016new}. SA-DR involves separating the index set $I:=\{1,\dots,N\}$ (where $N$ is as in \eqref{eqn:Nsets}) into strings along which the two-set DR operator is applied and taking a convex combination of the strings' endpoints to be the next iterate. Formally, letting $I_t:=(i_1^t,i_2^t,\dots,i_{\gamma(t)}^t)$ be an ordered, nonempty subset of $I$ with length $\gamma(t)$ for $t=1,\dots,M$ and $x_0 \in H$, set
	\begin{align*}
	x_{k+1}:=& \sum_{t=1}^M w_t \mathbb{V}_t(x_k) \quad \text{with} \quad w_t>0\;\; (\forall t=1,\dots,M)\;\; \text{and}\;\; \sum_{t=1}^Mw_t = 1 \\
	\text{where} \quad \mathbb{V}_t(x_k) :=& T_{i_{\gamma(t)}^t,i_1^t}T_{i_{\gamma(t)}^t-1,i_{\gamma(t)}^t}\dots T_{i_2^t,i_3^t}T_{i_1^t,i_2^t}(x_k),
	\end{align*}
	where $T_{A,B}$ is the 2-set DR operator. The principal result is as follows.
	\begin{theorem}{SA-DR (Censor \& Mansour, 2016 \cite[Theorem 18]{censor2016new})}
		Let $S_1,\dots,S_N \subset H$ be closed and convex with ${\rm int} \cap_{i \in I}S_i \neq \emptyset$. Then for any $x_0 \in H$, any sequence $(x_k)_{k=1}^\infty$ generated by the SA-DR algorithm with strings satisfying $I=I_1 \cup I_2 \cup \dots \cup I_M$ converges strongly to a point $x^* \in \cap_{i \in I}S_i$.
	\end{theorem}
	
	The BI-DR algorithm involves separating $I$ into subsets and applying 2-set DR to each of them by choosing a block index according to the rule $t_k = k \mod M +1$ and setting	
	\begin{align*}
	x_{k+1}:=& \sum_{j=1}^{\gamma({t_k})} w_j^{t_k} z_j \quad \text{with} \quad w_{t_k}>0\;\; (\forall j=1,\dots,\gamma({t_k}))\;\; \text{and}\;\; \sum_{j=1}^{\gamma({t_k})} w_j^{t_k} = 1, \\
	\text{where}\;\; z_j :=& T_{i_j^{t_k},i_{j+1}^{t_k}}(x_k)\;\;(\forall j=1,\dots,\gamma(t_k)-1) \;\; \text{and}\;\; z_{\gamma(t_k)}:= T_{i_{\gamma({t_k})}^{t_k},i_1^{t_k}}(x_k).
	\end{align*}
	The principal result is as follows. 
	\begin{theorem}{BI-DR (Censor \& Mansour, 2016 \cite[Theorem 19]{censor2016new})}
		Let $S_1,\dots,S_N \subset H$ be closed and convex with $\cap_{i \in I}S_i \neq \emptyset$. For any $x_0 \in H$, the sequence $(x_k)_{k=1}^\infty$ of iterates generated by the BI-DR algorithm with $I=I_1 \cup \dots \cup I_M$, after full sweeps through all blocks, converges
		\begin{enumerate}
			\item weakly to a point $x^*$ such that $P_{S_{i_j^t}}(x^*) \in \cap_{j=1}^{\gamma(t)} S_{i_j^t}$ for $j=1,\dots,\gamma(t)$ and $t=1,\dots,M$, and
			\item strongly to a point $x^*$ such that $x^* \in \cap_{i=1}^N S_i$ if the additional assumption ${\rm int} \cap_{i \in I}S_i \neq \emptyset$ holds.
		\end{enumerate} 
	\end{theorem}
	
	\subsubsection{Cyclic $r$-sets DR: Arag\'on Artacho-Censor-Gibali Method}
	
	Motivated by the intuition of the Borwein-Tam method and the example in Figure~\ref{fig:cyclicfail}, Francisco Arag\'on Artacho, Censor, and Aviv Gibali have recently introduced another method which simplifies to classical Douglas--Rachford method in the $2$-set case \cite[Theorem 3.7]{censor2018cyclic}. 
	
	For the feasibility problem of $N$ sets $S_0, \dots, S_{N-1}$, we denote by $S_{N,r}(d)$ the finite sequence of sets:
	$$
	S_{N,r}(d):= S_{\left((r-1)d-(r-1)\right){\rm mod}\; N}, S_{\left((r-1)d-(r-2)\right){\rm mod}\; N},\dots,S_{\left((r-1)d\right){\rm mod}\; N}.
	$$
	The method is then given by
	\begin{align*}
	x_{n+1}\;\;:=&\;\; \mathbb{V}_{N}\mathbb{V}_{N-1}\dots \mathbb{V}_{1}(x_n).\\
	\text{where}\quad \mathbb{V}_d \;\;:=&\;\; \frac{1}{2}\left(\Id + V_{C_{m,r}(d)} \right)\\
	\text{and}\quad  V_{C_0,C_1,\dots,C_{r-1}}\;\;:=&\;\;R_{C_{r-1}}R_{C_{r-2}}R_{C_0}.
	\end{align*}
	Provided the condition ${\rm int}\left(\cap_{i=0}^{N-1}S_i \right) \neq \emptyset$, the sequence $(x_n)_{n=1}^\infty$ converges weakly to a solution of the feasibility problem. They also provided a more general result, \cite[Theorem 3.4]{censor2018cyclic}, whose sufficiency criteria are, generically, more difficult to verify.
	
	\subsubsection{Sums of $N$ Operators: Spingarn's Method}
	
	One popular method for finding a point in the zero set of a sum of $N$ monotone operators $T_1,\dots,T_N$ is the reduction to a $2$ operator problem given by
	\begin{align}
	\A :=& T_1 \otimes T_2 \otimes \dots \otimes T_N\\
	\B :=& N_B
	\end{align}
	where $N_B$ is the normal cone operator \eqref{normalconeoperator} for $B$ and $B$ is the \emph{agreement set} defined in \eqref{eqn:B}. As $\A$ and $\B$ are maximal monotone, the weak convergence result is given by Svaiter's relaxation \cite{Svaiter} of Theorem~\ref{thm:LionsandMercier}. The application of DR to this problem is analogous to the product space method discussed in \ref{sec:productspace}. In 2007 Eckstein and Svaiter \cite{eckstein2009general} described this as \emph{Spingarn's method}, referencing Spingarn's 1983 article \cite{spingarn1983partial}. They also established a general projective framework for such problems which does not require reducing the problem to the case $N=2$.
	
	\section{Convex Setting}\label{convex_setting}
	
	Throughout the rest of the exposition, we will take the Douglas--Rachford operator and sequence to be as in \eqref{DRsequence}. Where no mention is made of the parameters $\lambda,\mu,\gamma$, it is understood that they are as in Definition~\ref{def:DRMethod}. While Theorems~\ref{thm:BCL}~and~\ref{thm:LionsandMercier} guarantee weak convergence for DR in the convex setting, only in finite dimensions is this sufficient to guarantee strong convergence. An important result of Hundal shows that AP may not converge in norm for the convex case when $H$ is infinite dimensional \cite{Hundal} (see also \cite{bauschke2005new,matouvskova2003hundal}). Although no analogue of Hundal's example seems known, to date, for DR in the infinite dimensional case norm convergence has been verified under additional assumptions on the nature of $A$ and $B$.
	
	\subsection{Convergence}\label{convex_convergence}
	
	Borwein, Li, and Tam \cite{BLT2015} attribute the first convergence rate results for DR to Hesse, Luke, and Patrick Neumann who in 2014 showed local linear convergence in the possibly nonconvex context of sparse affine feasibility problems \cite{hesse2014alternating}. Bauschke, Bello Cruz, Nghia, Phan, and Wang extended this work by showing that the rate of linear convergence of DR for subspaces is the cosine of the Friedrichs angle \cite{BCNPW}.
	
	In 2014, Bauschke, Bello Cruz, Nghia, Phan, and Wang \cite{bauschke2016optimal} used the convergence rates of matrices to find optimal convergence rates of DR for subspaces with more general averaging parameters as in \eqref{DRsequence}. In 2017, Mattias F\"{a}lt and Giselsson characterized the parameters that optimize the convergence rate in this setting \cite{falt2017optimal}.
	
	In 2014, Pontus Giselsson and Stephen Boyd demonstrated methods for preconditioning a particular class of problems with linear convergence rate in order to optimize a bound on the rate \cite{giselsson2014diagonal}.
	
	Motivated by the recent local linear convergence results in the possibly nonconvex setting \cite{HL13,hesse2014alternating,Phan,LP}, Borwein, Li, and Tam asked whether a global convergence rate for DR in finite dimensions might be found for a reasonable class of convex sets even when the regularity condition $\text{ri}A \cap \text{ri}B \ne \emptyset$ is potentially not satisfied. They provided some partial answers in the context of H\"{o}lder regularity with special attention given to convex semi-algebraic sets \cite{BLT2015}.
	
	Borwein, Sims, and Tam established sufficient conditions to guarantee norm convergence in the setting where one set is the positive Hilbert cone and the other set a closed affine subspace which has finite codimension \cite{BST2015}.
	
	In 2015, Bauschke, Dao, Noll, and Phan studied the setting of $\mathbb{R}^2$ where one set is the epigraph of a convex function and the other is the real axis, obtaining various convergence rate results \cite{BDNP16a}. In their follow-up article in 2016, they demonstrated finite convergence when Slater's condition holds in both the case where one set is an affine subspace and the other a polyhedron and in the case where one set is a hyperplane and the other an epigraph \cite{BDNP16b}. They included an analysis of the relevance of their results in the product space setting of Spingarn \cite{spingarn1983partial} and numerical experiments comparing the performance of DR and other methods for solving linear equations with a positivity constraint. In the same year, Bauschke, Dao, and Moursi provided a characterization of the behaviour of the sequence $(T^n x - T^n y)_{n \in \mathbb{N}}$ \cite{BDM2015}.
	
	In 2015, Damek Davis and Wotao Yin showed that DR might converge arbitrarily slowly in the infinite dimensional setting \cite{davis2016convergence}.
	
	\subsubsection{Order of operators}
	
	In 2016, Bauschke and Moursi investigated the order of operators: $T_{A,B}$ vs. $T_{B,A}$. In so doing, they demonstrated that $R_A: \Fix T_{A,B} \rightarrow \Fix T_{B,A}$ and $R_B: \Fix T_{B,A} \rightarrow \Fix T_{A,B}$ are bijections \cite{HMorder}.
	
	\subsubsection{Best approximations and the possibly infeasible case}
	
	The behaviour of DR in the inconsistent setting is most often studied using the minimal displacement vector
	\begin{equation}\label{eqn:displacementvector}
	v:=P_{\overline{\text{ran}}(\Id-T_{A,B})}0.
	\end{equation}
	The set of best approximation solutions relative to $A$ is $A\cap(v+B)$; when it is nonempty, the following have also been shown.
	
	In 2004 Bauschke, Combettes, and Luke considered the algorithm under the name \emph{averaged alternating reflections (AAR)}. They demonstrated that in the possibly inconsistent case, the shadow sequence $P_A x_n$ remains bounded with its weak sequential cluster points being in $A\cap(v+B)$ \cite{DRConvergence}.
	
	In 2015, Bauschke and Moursi \cite{HMaffine} analysed the more specific setting of two affine subspaces, showing that $P_A x_n$ will converge to a best approximation solution. In 2016, Bauschke, Minh Dao, and Moursi \cite{BDMaffine} furthered this work by considering the affine-convex setting, showing that when one of $A$ and $B$ is a closed affine subspace $P_A x_n$ will converge to a best approximation solution. They then applied their results to solving the least squares problem of minimizing $\sum_{k=1}^M d_{C_k}(x)^2$ with Spingarn's splitting method \cite{spingarn1983partial}.
	
	In 2016, Bauschke and Moursi provided a more general sufficient condition for the weak convergence \cite{BMshadow}, and in 2017 they characterized the magnitudes of minimal displacement vectors for more general compositions and convex combinations of operators.
	
	\subsubsection{Nearest feasible points (Arag\'{o}n Artacho-Campoy Method)}
	
	In 2017, Arag\'{o}n Artacho and Campoy introduced what they called the \emph{Averaged Alternating Modified Reflections (AAMR)} method for finding the nearest feasible point for a given starting point \cite{FranNewMethod}. The operator and method are defined with parameters $\alpha, \beta \in ]0,1[$ by
	\begin{align}
	T_{A,B,\alpha,\beta} &:= (1-\alpha)\Id + \alpha(2\beta P_B-\Id)(2\beta P_A -\Id) \nonumber\\
	x_n&:=T_{A-q,B-q,\alpha,\beta} x_n, n=0,1,\dots \label{AAMR}
	\end{align}
	which we recognize as DR in the case $\alpha=1/2,\beta=1,q=0$. The convergence result is as follows.
	\begin{theorem}{Arag\'{o}n Artacho \& Campoy 2017, \cite[Theorem 4.1]{FranNewMethod}}
		Given $A,B$ closed and convex, $\alpha,\beta \in ]0,1[$, and $q \in H$, choose any $x_0 \in H$. Let $(x_n)_{n \in \mathbb{N}}$ be as in \eqref{AAMR}. Then if $A\cap B \neq \emptyset$ and $q-P_{A\cap B}(q) \in (N_A+N_B)(P_{A\cap B}(q))$ then the following hold:
		\begin{enumerate}
			\item $(x_n)_{n \in \mathbb{N}}$ is weakly convergent to a point $x \in \Fix T_{A-q,B-q,\alpha,\beta}$ such that $P_A(q+x) = P_{A\cap B}(q)$;
			\item $(x_{n+1}-x_n)_{n \in \mathbb{N}}$ is strongly convergent to $0$;
			\item $(P_A(q+x_n))_{n \in \mathbb{N}}$ is strongly convergent to $P_{A\cap B}(q)$.
		\end{enumerate}
		Otherwise $\|x_n\|\rightarrow \infty$. Moreover, if $A,B$ are closed affine subspaces, $A\cap B \neq \emptyset$, and $q-P_{A\cap B}(q) \in (A-A)^\perp + (B-B)^\perp$ then $(x_n)_{n \in \mathbb{N}}$ is strongly convergent to $P_{\Fix T_{A-q,B-q,\alpha,\beta}}(x_0)$.
	\end{theorem}
	The algorithm may be thought of as another approach to the convex optimization problem of minimizing the convex function $y\mapsto \|q-y\|^2$ subject to constraints on the solution. 
	
	It is quite natural to consider the theory of the algorithm in the case where projection operators $P_A = J_{N_A},P_B = J_{N_B}$ are replaced with more general resolvents for maximally monotone operators \cite{artacho2018computing}, an extension Arag\'{o}n Artacho and Campoy gave in 2018. This work has already been extended by Alwadani, Bauschke, Moursi, and Wang, who analysed the asymptotic behaviour and gave the algorithm the more specific name of \emph{Arag\'{o}n Artacho-Campoy Algorithm (AACA)} \cite{Alwadani}.
	
	\subsubsection{Cutter methods}
	
	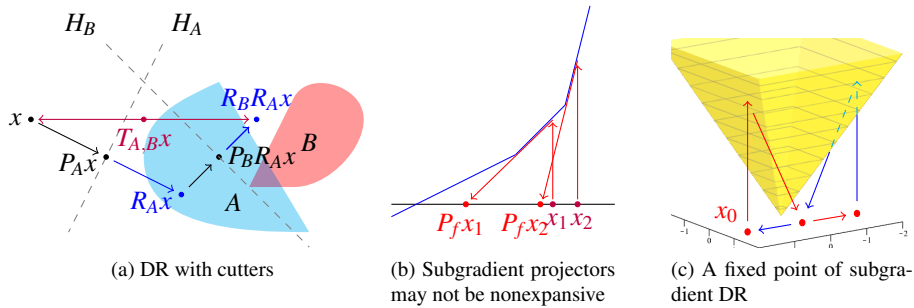
\begin{figure}[h]
		\begin{center}
			\subfloat[DR with cutters]{\begin{tikzpicture}[scale=2.0]
				\fill[cyan,fill opacity=0.4] (1,1) to (1.6,0)
				to [out=180,in=270] (.5,.5)
				to [out=90,in=205] cycle;
				\node at (1.1,.2) {$A$};
				
				\fill[purple,fill opacity=0.4] (1.42,.3) to (1.2,.3) to (1.3,.5) to cycle;
				
				\fill[red,fill opacity=0.4] (1.42,.3) to (1.3,.5) to (1.5,.9) to [out=60,in=135] (1.9,.9) to [out=315,in=0] (1.5,.3) to cycle;
				\node [] at (1.6,.6) {$B$}; 					
				
				
				\draw [gray, dashed] (.625,1.25) -- (0,0);
				\node [above right] at (.625,1.25) {$H_A$};
				
				\draw [fill,black] (-.25,.75) circle [radius=0.015];	
				\node [left] at (-.25,.75) {$x$};
				\draw [black,->] (-.2,.725) -- (.2,.525);
				\draw [fill,black] (.25,.5) circle [radius=0.015];
				\node [left] at (.25,.45) {$P_Ax$};
				\draw [blue,->] (.3,.475) -- (.7,.275);
				\draw [fill,blue] (.75,.25) circle [radius=0.015];	
				\node [left, blue] at (.75,.2) {$R_Ax$};

				\draw [blue,->] (1.05,.55) -- (1.2,.7);
				\draw [fill,blue] (1.25,.75) circle [radius=0.015];
				\node [above,blue] at (1.25,.75) {$R_B R_A x$};
				\draw [black,->] (.8,.3) -- (.95,.45);
				\draw [fill,black] (1,.5) circle [radius=0.015];
				\draw [blue,->] (1.05,.55) -- (1.2,.7);
				\node [right] at (1,.5) {$P_B R_A x$};
				
				\draw [gray, dashed] (.25,1.25) -- (1.6,-.1);
				\node [above left] at (.25,1.25) {$H_B$};
				
				\draw [purple, <->] (-.2,.75) -- (1.2,.75);
				\draw [fill,purple] (.5,.75) circle [radius=0.015];
				\node [below,purple] at (.5,.75) {$T_{A,B}x$};
				
				\end{tikzpicture}\label{fig:cuttersleft}}\hspace*{0.5\fill}
			\subfloat[Subgradient projectors may not be nonexpansive]{\begin{tikzpicture}[scale=0.66]
				\draw [black] (-.5,0) -- (4,0);
				\draw [blue] (-.5,-.25) -- (2,1);
				\draw [blue] (2,1) -- (3,2);
				\draw [blue] (3,2) -- (3.5,4);
				
				\draw [fill,purple] (3.25,0) circle [radius=0.05];
				\node [below,purple] at (3.35,0) {$x_2$};
				\draw [red,->] (3.25,.1) -- (3.25,2.8);
				\draw [red,->] (3.25,2.9) -- (2.55,.1);					
				\draw [fill,red] (2.5,0) circle [radius=0.05];
				\node [below, red] at (2.2,0) {$P_f x_2$};			
				
				\draw [fill,purple] (2.75,0) circle [radius=0.05];
				\node [below,purple] at (2.85,0) {$x_1$};				
				\draw [red,->] (2.75,.1) -- (2.75,1.65);
				\draw [red,->] (2.75,1.7) -- (1.1,.1);			
				\draw [fill,red] (1,0) circle [radius=0.05];
				\node [below,red] at (.9,0) {$P_f x_1$};							
				\end{tikzpicture}\label{fig:cutterscenter}}\hspace*{\fill}
			\subfloat[A fixed point of subgradient DR]{\begin{adjustbox}{trim=0.3cm 0.5cm 0cm 0.8cm,clip=true}
					\begin{tikzpicture}[scale=1]
					\node[anchor=south west,inner sep=0] (image) at (0,0) {\includegraphics[width=.29\textwidth]{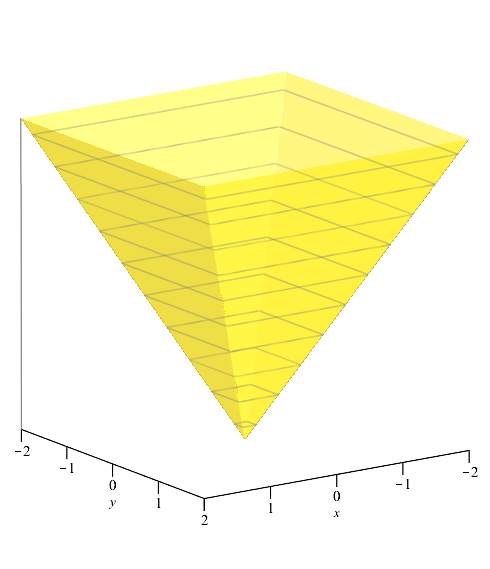}};
					\begin{scope}[x={(image.south east)},y={(image.north west)}]
					\draw [fill,red] (.375,.175) circle [radius=0.01];
					\node [above left,red] at (.375,.175) {$x_0$};
					
					\draw [fill,red] (.581,.205) circle [radius=0.01];
					\draw [red,->] (.375,.21) -- (.375,.6);
					\draw [red,->] (.396,.561) -- (.56,.245);
					
					\draw [fill,red] (.787,.235) circle [radius=0.01];
					\draw [red,->] (.622,.211) -- (.746,.229);
					
					\draw [blue] (.787,.275) -- (.787,.555);
					\draw [cyan,dashed,->] (.787,.555) -- (.787,.665);
					\draw [cyan,dashed] (.766,.619) -- (.684,.435);
					\draw [blue,->] (.684,.435) -- (.602,.251);
					
					\draw [blue,->] (.539,.199) -- (.416,.181);
					
					\end{scope}
					\end{tikzpicture}
				\end{adjustbox}\label{fig:cuttersright}}
		\end{center}
		\caption{Cutter methods}\label{fig:cutters}
	\end{figure}
	Another computational approach is to replace true projections with approximate projections or \emph{cutter} projections onto separating hyperplanes as in Figure~\ref{fig:cuttersleft}. Prototypical of this category are subgradient projection methods which may be used to find $x \in \cap_{i=1}^m \text{lev}_{\leq 0}f_i$ for $m$ convex functions $f_1,\dots,f_m$; see Figure~\ref{fig:cutterscenter}. Such methods are not generally nonexpansive (as shown in \ref{fig:cutterscenter}) but may be easier to compute. When true reflection parameters are allowed, the method is no longer immune from ``bad'' fixed points, as illustrated in Figure~\ref{fig:cuttersright}. However, with a suitable restriction on reflection parameters and under other modest assumptions, convergence may be guaranteed through Fej\'{e}r monotonicity methods. See, for example, the works of Cegielski and Fukushima \cite{Cegielski,Fukushima}. More recently, D\'{i}az Mill\'{a}n, Lindstrom, and Roshchina have provided a standalone analysis of DR with cutter projections \cite{DLR18}.
	
	\subsection{Notable Applications}\label{convex_applications}
	
	While the Douglas--Rachford operator is firmly nonexpansive in the convex setting, the volume of literature about it is expansive indeed. While reasonable brevity precludes us from furnishing an exhaustive catalogue, we provide a sampling of the relevant literature.
	
	As early as 1961, working in the original context of Douglas and Rachford, P.L.T. Brian introduced a modified version of DR for high order accuracy solutions of heat flow problems \cite{brian1961finite}.
	
	In 1995, Fukushima applied DR to the traffic equilibrium problem, comparing its performance (and the complexity and applicability of the induced algorithms) to ADMM \cite{fukushima1996primal}.
	
	In 2007, Combettes and Jean-Christophe Pesquet applied a DR splitting to nonsmooth convex variational signal recovery, demonstrating their approach on image denoising problems \cite{combettes2007douglas}.
	
	In 2009, Simon Setzer showed that the \emph{Alternating Split Bregman Algorithm} from \cite{goldstein2009split} could be interpreted as a special case of DR in order to interpret its convergence properties, applying the former to an image denoising problem \cite{setzer2009split}. In the same year, Gabriele Steidl and Tanja Teuber applied DR for removing multiplicative noise, analysing its linear convergence in their context and providing computational examples by denoising images and signals \cite{steidl2010removing}.
	
	In 2011 Combettes and Jean-Christophe Pesquet contrasted and compared various proximal point algorithms for signal processing \cite{combettes2011proximal}.
	
	In 2012 Laurent Demanet and Xiangxiong  Zhang applied DR to $l_1$ minimization problems with linear constraints, analysing its convergence and bounding the convergence rate in the context of compressed sensing \cite{demanet2016eventual}.		
	
	In 2012, Radu Ioan Bo\c{t}, and Christopher Hendrich proposed two algorithms based on Douglas--Rachford splitting, which they used to solve a generalized Heron problem and to deblur images \cite{bot2013douglas}. In 2014, they analysed with Ern\"{o} Robert Csetnek an \emph{inertial} DR algorithm and used it to solve clustering problems \cite{boct2015inertial}. 
	
	In 2015, Bauschke, Valentin Koch, and Phan applied DR for a road design optimization problem in the context of minimizing a sum of proper convex lsc functions, demonstrating its effectiveness on real world data \cite{bauschke2015stadium}.
	
	In 2017, Fei Wang, Greg Reid, and Henry Wolkowicz applied DR with facial reduction for a set of matrices of a given rank and a linear constraint set in order to find maximum rank moment matrices \cite{wang2017finding}. 
	
	\section{Non-convex Setting}\label{nonconvex}
	
	Investigation in the nonconvex setting has been two-pronged, with the theoretical inquiry chasing the experimental discovery. The investigation has also taken place in, broadly, two contexts: that of curves and/or hypersurfaces and that of discrete sets.
	
	While Jonathan Borwein's exploration spanned both of the aforementioned contexts, his interest in DR appears to have been initially sparked by its surprising performance in the latter \cite{Elseremail}, specifically the application of the method by Elser, Rankenburg, and Thibault to solving a wide variety of combinatorial optimization problems, including Sudoku puzzles \cite{ERT07}. Where the product space reformulation is applied to feasibility problems with discrete constraint sets, DR often succeeds while AP does not. Early wisdom suggested that one reason for its superior performance is that DR, unlike AP, is immune from false fixed points regardless of the presence or absence of convexity, as shown in the following proposition (see, for example, \cite[Proposition 1.5.1]{Tam} or \cite{ERT07}).
	\begin{proposition}[Fixed points of DR]\label{prop:fixedpoint} Let $A,B \subset H$ be proximal. Then $x \in \Fix T_{A,B}$ implies $P_A(x) \in A \cap B$. 
	\end{proposition}
	\begin{proof}
		Let $x \in \Fix T_{A,B}$. Then $x = x+P_B(2P_A(x)-x)-P_A(x)$ and so $P_B(2P_A(x)-x)-P_A(x) =0$, so $P_A(x) \in B$.
	\end{proof}
	A typical example where $A:=\{a_1,a_2\}$ is a doubleton and $B$ a subspace (analogous to the agreement set) is illustrated in Figure~\ref{fig:DRvsAP} where DR is seen to solve the problem while AP becomes trapped by a fixed point.
	\begin{figure}[ht]
		\begin{center}
			\subfloat[DR]{
				\begin{tikzpicture}
				[scale=0.95]
				\node[anchor=south west,inner sep=0] (image) at (0,0) {\includegraphics[width=.3\linewidth,angle=90]{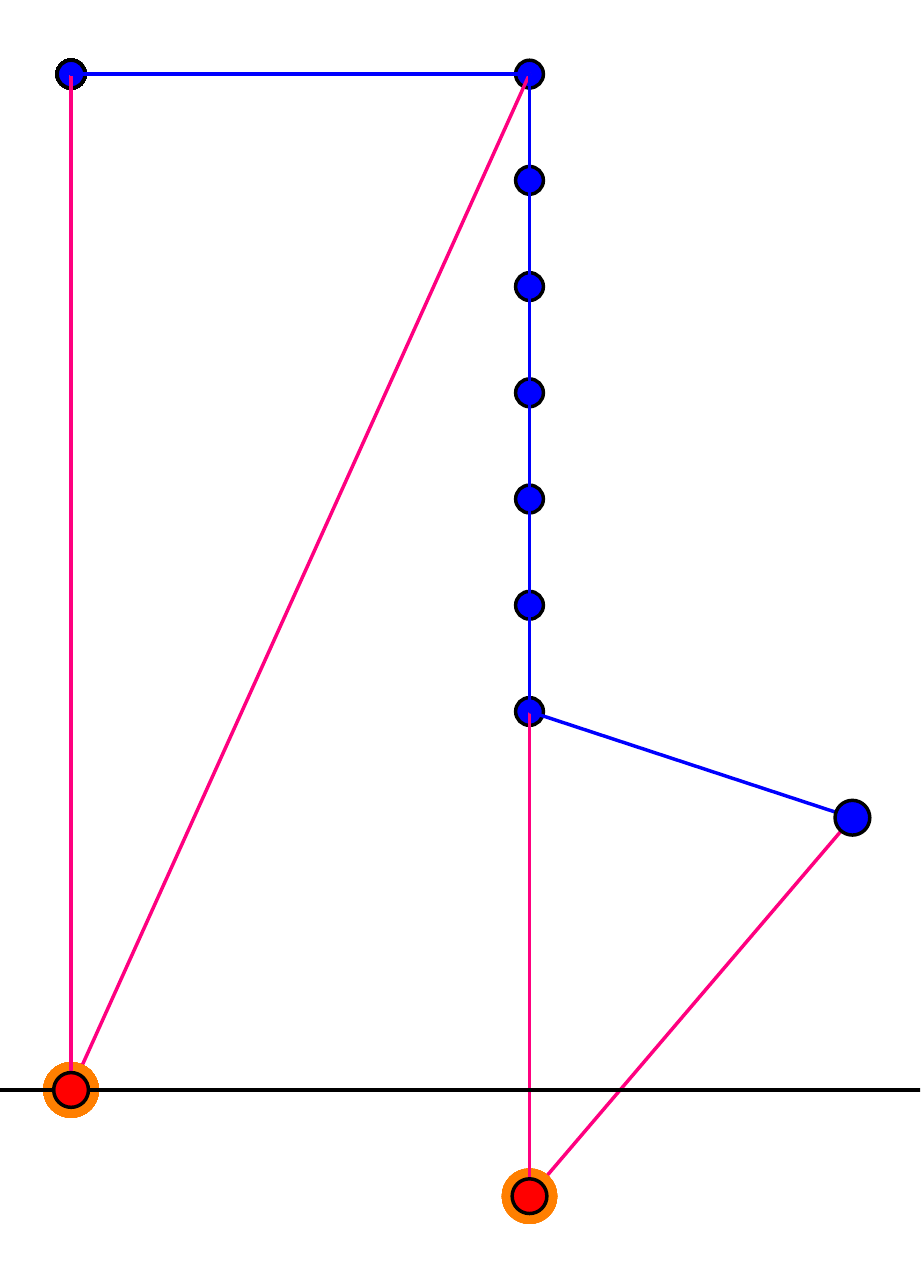}};
				\begin{scope}[x={(image.south east)},y={(image.north west)}]
				
				\node [left,blue] at (.64,.92) {$x_0$};
				
				\node [below,blue] at (.56,.58) {$x_1$};
				
				\node [above,blue] at (.06,.58) {$x_7$};	
				
				\node [below right,blue] at (0.0,.08) {$x_n \; \text{for}\; n\geq 8$};	
				
				\node [below left,orange] at (0.85,.08) {$P_A x_n \; \text{for}\; n\geq 8$};
				
				\node [right,red] at (0.87,.08) {$a_2$};
				
				\node [below,red] at (0.95,.56) {$a_1$};
				
				\node [right] at (0.85,.8) {$B$};
				
				\end{scope}
				\end{tikzpicture}\label{fig:DRvsAPleft}}\hspace{1cm}
			\subfloat[AP]{
				\begin{tikzpicture}
				[scale=0.95]
				\node[anchor=south west,inner sep=0] (image) at (0,0) {\includegraphics[width=.3\linewidth,angle=90]{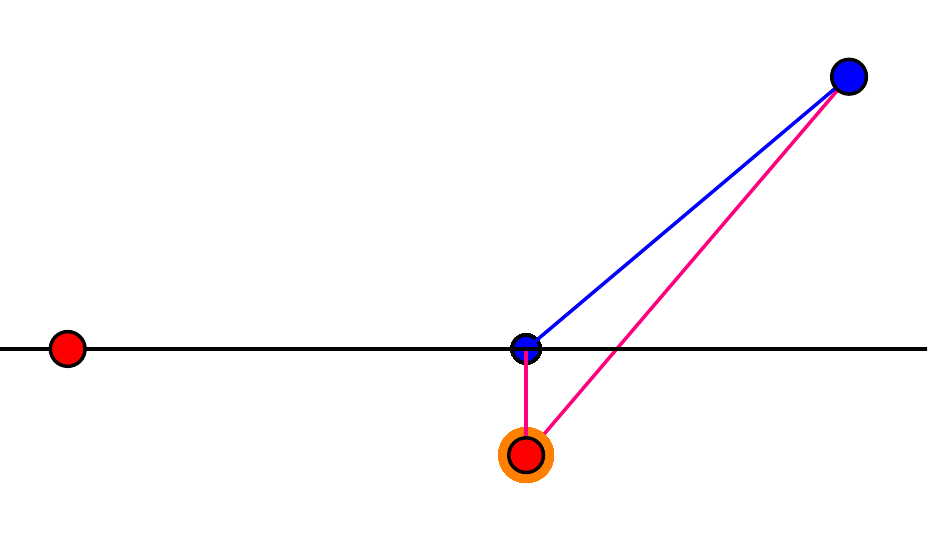}};
				\begin{scope}[x={(image.south east)},y={(image.north west)}]
				
				\node [right] at (0.65,.8) {$B$};
				
				\node [left,blue] at (0.1,.9) {$x_0$};
				
				\node [left,blue] at (0.65,.55) {$x_n \; \text{for}\; n\geq 1$};
				
				\node [below right,red] at (0.75,.55) {$a_1$};
				
				\node [right,red] at (0.65,.09) {$a_2$};
				
				\node [below left,orange] at (0.85,.07) {$\textcolor{white}{R}$};
				\end{scope}
				\end{tikzpicture}}
		\end{center}
		\caption{DR and AP for a doubleton and a line in $\mathbb{R}^2$}\label{fig:DRvsAP}
	\end{figure}
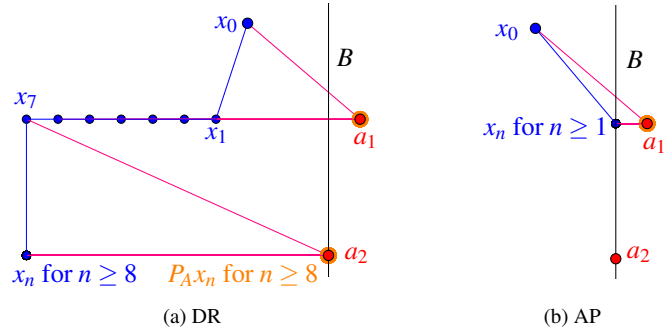
	
	If the germinal work on DR in the nonconvex setting is that of Elser, Rankenburg, and Thibault \cite{ERT07} (caution: the role of $A$ and $B$ are reversed from those here), then the seminal work is that of J.R. Fienup who applied the method to solve the phase retrieval problem \cite{fienup1982phase}. In \cite{ERT07}, Elser et al. referred to DR as \emph{Fienup's iterated map} and the \emph{difference map}, while Fienup himself called it the \emph{hybrid input-output algorithm} (HIO) \cite{fienup1982phase}. Elser explains that, originally, neither Fienup nor Elser et al. were aware of the work of Lions and Mercier \cite{LM}, and so the seminal work on DR in the nonconvex setting is, surprisingly, an independent discovery of the method \cite{Elseremail}. Fienup constructed the method by combining aspects of two other methods he considered---the \emph{basic input-output} algorithm and the \emph{output-output} algorithm---with the intention of obviating stagnation. Here again, one may think of the behaviour illustrated in Figure~\ref{fig:DRvsAP}. 
	\begin{figure}[ht]
		\hspace*{\fill}\subfloat[Convergence to a feasible point]{\includegraphics[width=.3\textwidth]{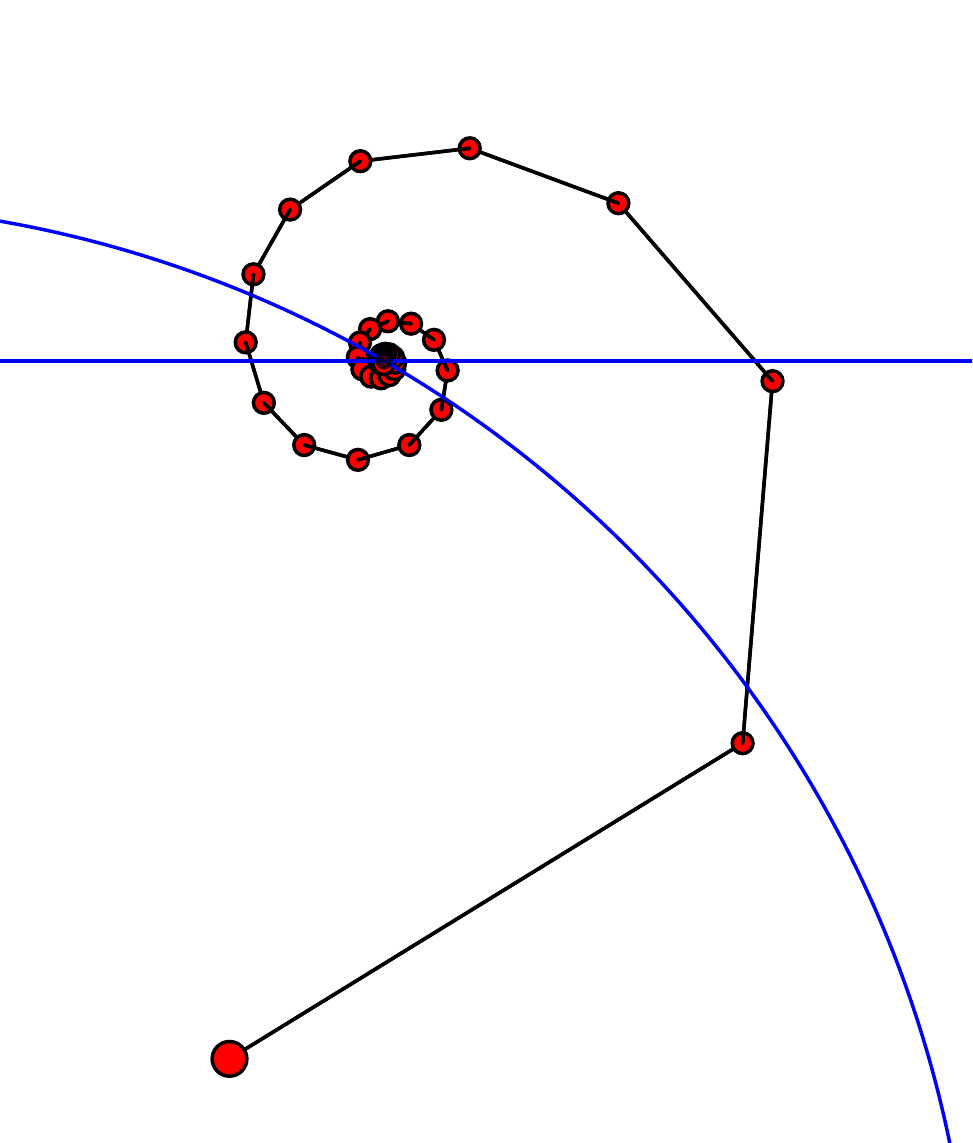}\label{fig:circleandlineleft}}\hspace*{\fill}\subfloat[Convergence to a fixed point]{\includegraphics[width=.3\textwidth]{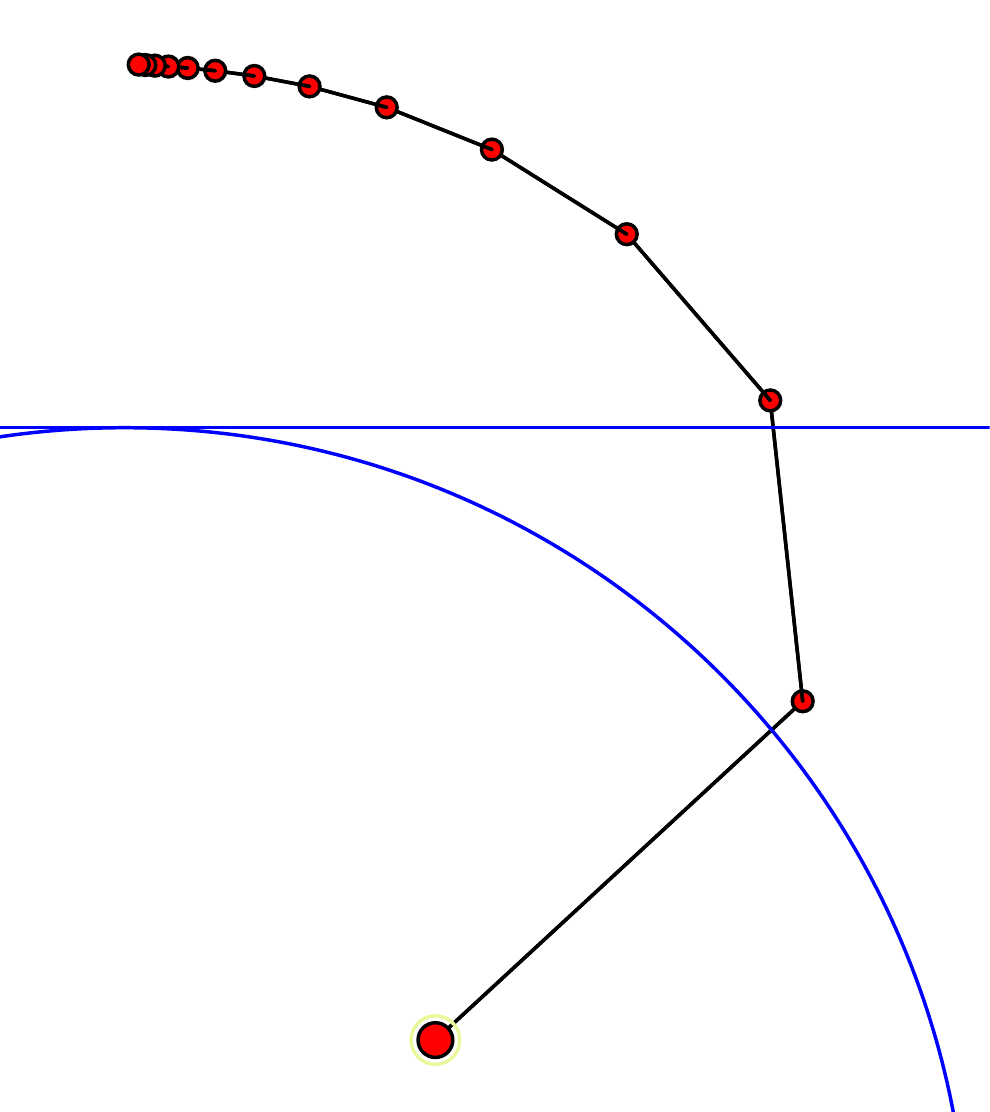}\label{fig:circleandlinecenter}}\hspace*{\fill}\subfloat[Proof of convergence with Benoist's Lyapunov function]{\includegraphics[width=.3\textwidth]{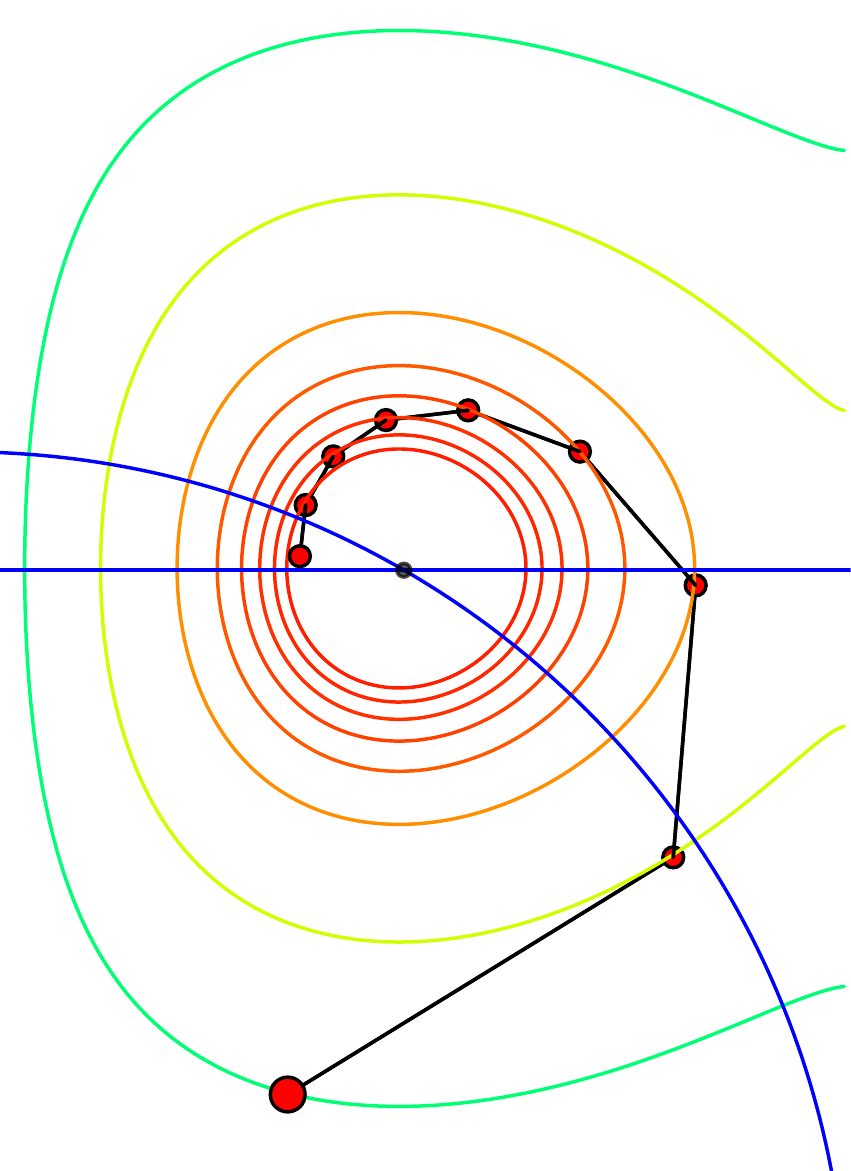}\label{fig:circleandlineright}}\hspace*{\fill}
		\caption{Behaviour of DR where $A$ is a circle and $B$ is a line; (c) is discussed in \ref{nonconvex_hypersurfaces}. }\label{fig:circleandline}
	\end{figure}
	
	Figure~\ref{fig:circleandline} shows behaviour of DR in the case where $A$ is a circle and $B$ is a line, a situation prototypical of the phase retrieval problem. For most arrangements, DR converges to a feasible point as in Figure~\ref{fig:circleandlineleft}. However, when the line and circle meet tangentially as in Figure~\ref{fig:circleandlinecenter}, DR converges to a fixed point which is not feasible, and the sequence $P_A x_n$ converges to the true solution.
	
	Elser notes that it is unclear whether or not Fienup understood that a fixed point of the algorithm is not necessarily feasible, as his approach was largely empirical. Elser sought to clarify this point in his follow up article in which he augmented the study of DR for phase retrieval by replacing support constraints with object histogram and atomicity constraints for crystallographic phase retrieval \cite[Section 5]{elser2003phase}. In 2001 when \cite{elser2003phase} was submitted, Elser was not yet aware of Lions' and Mercier's characterization of DR as the \emph{Douglas--Rachford method}; it may be recognized in \cite{elser2003phase} as a special instance of the \emph{difference map} (which we define in \eqref{difference_map}), a generalization of \emph{Fienup's input-output map}.
	
	In 2002, Bauschke, Combettes, and Luke finally demonstrated that Fienup's basic input-output algorithm is an instance of Dykstra's algorithm and that HIO (Hybrid Input--Output) with the support constraint alone corresponds to DR \cite{BCL} (see also their 2003 follow up \cite{bauschke2003hybrid}). In another follow up \cite{DRConvergence}, they showed that with support and nonnegativity constraints, HIO corresponds to the HPR (hybrid projection reflection) algorithm, a point Luke sought to clarify in his succinct 2017 summary of the investigation of DR in the context of phase retrieval \cite{luke2017phase}.
	
	More recently, in 2017, Elser, Lan, and Bendory have published a set of benchmark problems for phase retrieval \cite{elser2017benchmark}. They considered DR with true reflections and a relaxed averaging parameter---$\mu=\gamma=0,\lambda \in ]0,1]$ as in \eqref{DRsequence}---under the name \emph{relaxed-reflect-reflect (RRR)}. In particular, they provided experimental evidence for the exponential growth of DR's mean iteration count as a function of the autocorrelation sparsity parameter, which seems well-suited for revealing behavioural trends. They also provided an important clarification of the different algorithms which have been labelled ``Fienup'' algorithms in the literature, some of which are not DR.
	
	\subsection{Discrete Sets}\label{nonconvex_discrete}
	
	The landmark experimental work on discrete sets is that of Elser, Rankenburg, and Thibault \cite{ERT07}. They considered the performance of what they called the \emph{difference map} for various values of the parameter $\beta$:
	\begin{align}
	T:x &\mapsto x + \beta \left(P_A\circ f_B(x)-P_B\circ f_A(x) \right), \label{difference_map}\\
	\text{where}\quad f_A:x &\mapsto P_A(x)-\left(P_A(x)-x \right)/\beta, \nonumber\\
	\text{and}\quad f_B:x &\mapsto P_B(x) + \left(P_B(x)-x \right)/\beta .\nonumber
	\end{align}
	When $\beta=-1$, we recover the DR operator $T_{A,B}$, and when $\beta=1$, we obtain $T_{B,A}$.
	
	\subsubsection{Stochastic Problems}
	
	Much of the surprising success of DR has been in the setting where some of the sets of interest have had the form $\{0,1\}^p$. Elser et al. adopted the approach of using stochastic feasibility problems to study the performance of DR \cite{ERT07}. They began with the problem of solving the linear Diophantine equation $Cx=b$, where $C$, a stochastic $p \times q$ matrix, and $b \in \mathbb{N}^p$ are ``known.'' Requiring the solution $x \in \{0,1\}^q$ that is used to generate the problem to also be stochastic ensures uniqueness of the solution for the feasibility problem: find $x \in A\cap B$ where 
	\begin{equation*}
	A:= \{0,1\}^q, \quad \text{and} \quad B:=\{x \in \mathbb{R}^q \;\text{such that}\; Cx=b \}.
	\end{equation*} 
	They continued by solving Latin squares of $n$ symbols. Where $x_{ijk}=1$ indicates that the cell in the $i$th row of the $j$th column of the square is $k$, the problem is stochastic and the constraint that $x_{\hat{i}\hat{j}\hat{k}}=1$ if and only if $(\forall i\neq \hat{i})\;x_{i\hat{j}\hat{k}} = 0 , (\forall j\neq \hat{j}) x_{\hat{i}j\hat{k}} = 0$ determines the set of allowable solutions. The most familiar form of a Latin square is the Sudoku puzzle where $n=9$ and we require the additional constraint that the complete square consist of a grid of $9$ smaller Latin squares. For more on the history of the application of projection algorithms to solving Sudoku puzzles, see Schaad's master's thesis \cite{schaad2010modeling} in which he also applies the method to the $8$ queens problem .
	
	This work of Elser et al. piqued the interest of Jonathan Borwein who in 2013, together with Arag\'{o}n Artacho and Matthew Tam, continued the investigation of Sudoku puzzles \cite{ABT1,ABT1b}, exploring the effect of formulation (integer feasibility vs. stochastic) on performance. They also extended the approach by solving nonogram problems.
	
	\subsubsection{Matrix completion and decomposition}
	
	Another application for which DR has shown promising results is finding the remaining entries of a partially specified matrix in order to obtain a matrix of a given type. Borwein, Arag\'{o}n Artacho, and Tam considered the behaviour of DR for such matrix completion problems \cite{ABT2}. They provided a discussion of the convex setting, including positive semidefinite matrices, correlation matrices, and doubly stochastic matrices. They went on to provide experimental results for a number of nonconvex problems, including for rank minimization, protein reconstruction, and finding Hadamard and skew-Hadamard matrices. In 2017, Artacho, Campoy, Kotsireas, and Tam applied DR to construct various classes of circulant combinatorial designs \cite{ACKT2018}, reformulating them as three set feasibility problems. Designs they studied included Hadamard matrices with two circulant cores, as well as circulant weighing matrices and D-optimal matrices.
	
	Even more recently, David Franklin used DR to find compactly supported, non-separable wavelets with orthonormal shifts, subject to the additional constraint of regularity \cite{Franklin,Franklinemail}. Reformulating the search as a three set feasibility problem in $\{\mathbb{C}^{2\times2}\}^M$ for $M=\{4,6,8,\dots\}$, they compared the performance of cyclic DR, product space DR, cyclic projections, and the \emph{proximal alternating linear minimisation} (or \emph{PALM}) algorithm. Impressively, product space DR solved every problem it was presented with.
	
	In 2017, Elser applied DR---under the name RRR (short for \emph{relaxed reflect-reflect})---for matrix decomposition problems, making several novel observations about DR's tendency to wander, searching in an apparently chaotic manner, until it happens upon the basin for a fixed point \cite{elser2017matrix}. These observations have motivated the open question we pose in \ref{complexity_theory}.
	
	\subsubsection{The study of proteins}
	
	In 2014, Borwein and Tam went on to consider protein conformation determination, reformulating such problems as matrix completion problems \cite{BT}. An excellent resource for understanding the early class problems studied by Borwein, Tam, and Arag\'{o}n Artacho---as well as the cyclic DR algorithm described in subsection~\ref{subsec:Nsets}---is Tam's PhD dissertation \cite{Tam}.
	
	Elser et al. applied DR to study protein folding problems, discovering much faster performance than that of the landscape sampling methods commonly used \cite{ERT07}. 
	
	\subsubsection{Where A is a subspace and B a restriction of allowable solutions}
	Elser et al. applied DR to the study of 3-SAT problems, comparing its performance to that of another solver, Walksat \cite{ERT07} (see also \cite{GE}). They found that DR solved all instances without requiring random restarts. They also applied the method to the spin glass ground state problem, an integer quadratic optimization program with nonpositive objective function.
	
	\subsubsection{Graph Coloring}
	Elser et al. applied DR to find colorings of the \emph{edges} of complete graphs with the constraint that no triangle may have all its edges of the same color \cite{ERT07}. They compared its performance to CPLEX, and included an illustration showing the change of edge colors over time. DR solved all instances, and outperformed CPLEX in harder instances.
	
	In 2016, Francisco Arag\'{o}n Artacho and Rub\'{e}n Campoy applied DR to solving graph coloring problems in the usual context of coloring \emph{nodes} \cite{AC}. They constructed the feasibility problem by attaching one of two kinds of gadgets to the graphs, and they compared performance with the two different gadget types both with and without the inclusion of maximal clique information. They also explored the performance for several other problems reformulated as graph coloring problems; these included: 3-SAT, Sudoku puzzles, the eight queens problem and generalizations thereof, and Hamiltonian path problems.
	
	More recently, Arag\'on Artacho, Campoy, and Elser \cite{artacho2018enhanced} have considered a reformulation of the graph coloring problem based on semidefinite programming, demonstrating its superiority through numerical experimentation.
	
	\subsubsection{Other implementations}
	Elser et al. went on to consider the case of bit retrieval, where $A$ is a Fourier magnitude/autocorrelation constraint and $B$ is the binary constraint set $\{\pm 1/2\}^n$ \cite{ERT07}. They found its performance to be superior to that of CPLEX. 
	
	Bansal used DR to solve Tetravex problems \cite{bansal2010code}.
	
	More recently, in 2018, Elser expounded further upon the performance of DR under varying degrees of complexity by studying its behaviour on bit retrieval problems \cite{elser2018complexity}. Of his findings of its performance  he observes:
	
	\begin{displayquote}
		These statistics are consistent with an algorithm that blindly and repeatedly reaches into an urn of $M$ solution candidates, terminating when it has retrieved one of the $4\times 43$ solutions. Two questions immediately come to mind. The easier of these is: How can an algorithm that is deterministic over most of its run-time behave randomly? The much harder question is: How did the $M=2^{43}$ solution candidates get reduced, apparently, to only about $1.7\times 10^5 \times (4 \times 43) \approx 2^{24}$?
	\end{displayquote}
	
	The behaviour of DR under varying complexity remains a fascinating open topic, and we provide it as one of our two open problems in \ref{complexity_theory}.
	
	\subsubsection{Theoretical Analysis}
	
	One of the first global convergence results in the nonconvex setting was given by Arag\'{o}n Artacho, Borwein, and Tam in the setting where one set is a half space and the second set finite \cite{FranMattJonGlobal}. Bauschke, Dao, and Lindstrom have since fully categorized the global behaviour for the case of a hyperplane and a doubleton (set of two points) \cite{BDL18}. Both problems are prototypical of discrete combinatorial feasibility problems, the latter especially insofar as the hyperplane is analogous to the agreement set in the product space version of the method discussed in section~\ref{sec:productspace}, the most commonly employed method for problems of more than $2$ sets.
	
	\subsection{Hypersurfaces}\label{nonconvex_hypersurfaces}	
	
	In 2011, Borwein and Sims made the first attempt at deconstructing the behaviour of DR in the nonconvex setting of hypersurfaces \cite{BS}. In particular, they considered in detail the case of a circle $A$ and a line $B$, a problem prototypical of phase retrieval. Here the dynamical geometry software \emph{Cinderella} \cite{cinderella2} first played an important role in the analysis: the authors paired Cinderella's graphical interface with accurate computational output from \emph{Maple} in order to visualize the behaviour of the dynamical system. Borwein and Sims went on to show local convergence in the feasible case where the line is not tangent to the $2$-sphere by using a theorem of Perron. They concluded by suggesting analysis for a generalization of the $2$-sphere: $p$-spheres.
	
	In 2013 Arag\'{o}n Artacho and Borwein revisited the case of a $2$-sphere and line intersecting non-tangentially \cite{AB}. When $x_0$ lies in the subspace perpendicular to $B$, the sequence $(x_n)_{n=0}^\infty$ is contained in the subspace and exhibits behaviour that at first appeared chaotic; it was later shown by Bauschke, Dao, and Lindstrom to be describable by generalized Beatty sequences \cite{BDL18}. For $x_0$ not in the aforementioned subspace---which we call the \emph{singular set}---Arag\'{o}n Artacho and Borwein provided a conditional proof of global convergence of iterates to the nearer of the two feasible points. The proof relied upon constructing and analysing the movement of iterates through different regions. Borwein humorously remarked of the result that, ``This was definitely not a proof from \emph{the book}. It was a proof from the \emph{anti-book}.'' Jo\"{e}l Benoist later provided an elegant proof of global convergence by constructing the Lyapunov function seen in Figure~\ref{fig:circleandlineright} \cite{Benoist}.
	
	In one of his later posthumous publications on the subject \cite{borwein2017ergodic}, Borwein, together with Ohad Giladi, demonstrated that the $DR$ operator for a sphere and a convex set may be approximated by another operator satisfying a weak ergodic theorem.
	
	\begin{figure}[ht]
		\begin{center}
			\includegraphics[angle=90,width=.75\textwidth]{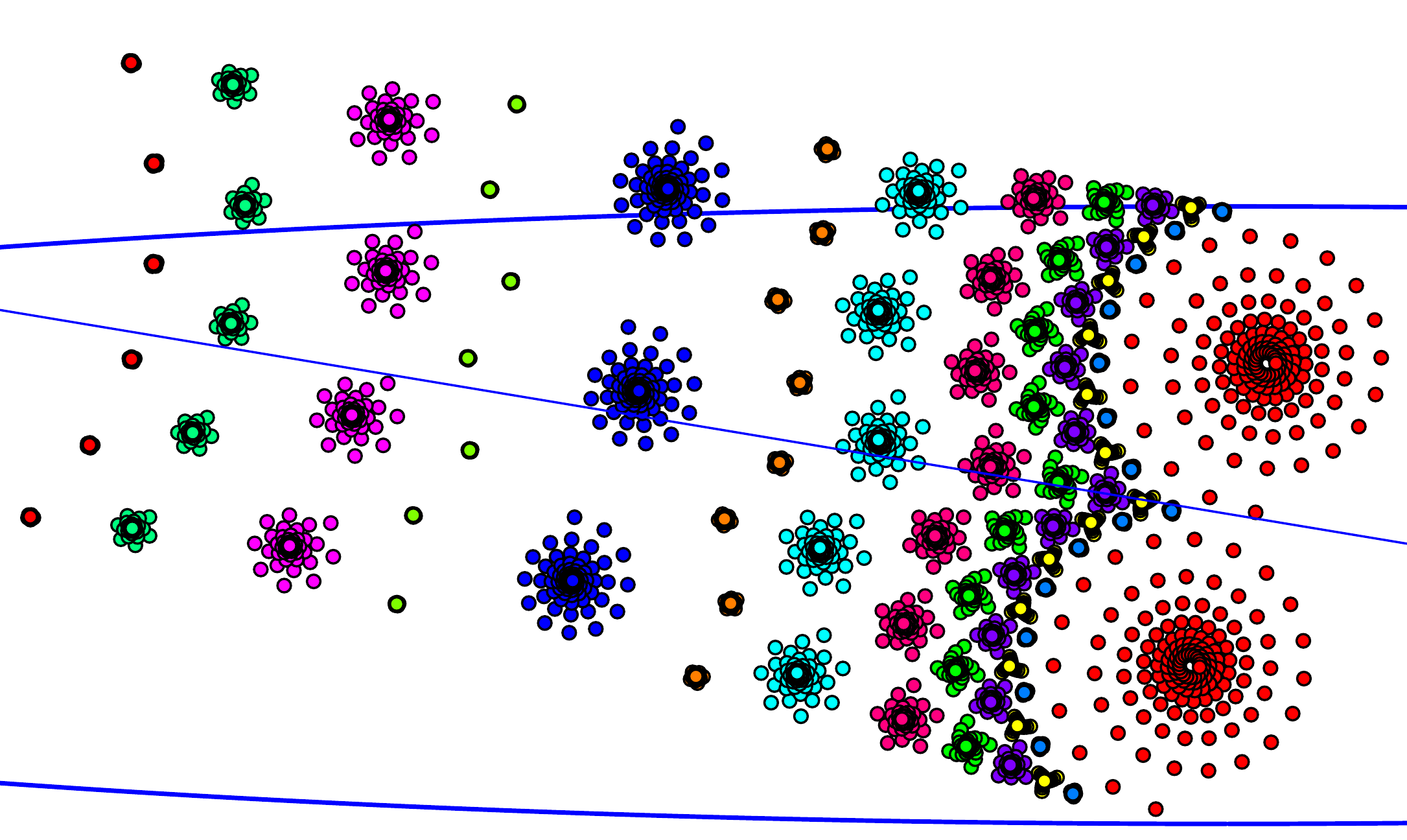}
		\end{center}
		\caption{Basins of attraction for periodic points with an ellipse and line}\label{fig:ellipsesequences}
	\end{figure}
	
	In 2016, Borwein, Lindstrom, Sims, Skerrit, and Schneider undertook Borwein's suggested follow up work in $\mathbb{R}^2$, analysing not only the case of $p$-spheres more generally but also of ellipses \cite{BLSSS}. They discovered incredible sensitivity of the global behaviour to small perturbations of the sets, with some arrangements eliciting a complex and beautiful geometry characterized by \emph{periodic points} with corresponding basins of attraction. A point $x$ satisfying $T^n x=x$ is said to be periodic with \emph{period} the smallest $n$ for which this holds; Figure~\ref{fig:ellipsesequences} from \cite{BLSSS} shows $13$ different DR sequences for an ellipse and line from which subsequences converge to periodic points.
	\begin{figure}[ht]
		\begin{center}
			\includegraphics[angle=90,width=.75\textwidth]{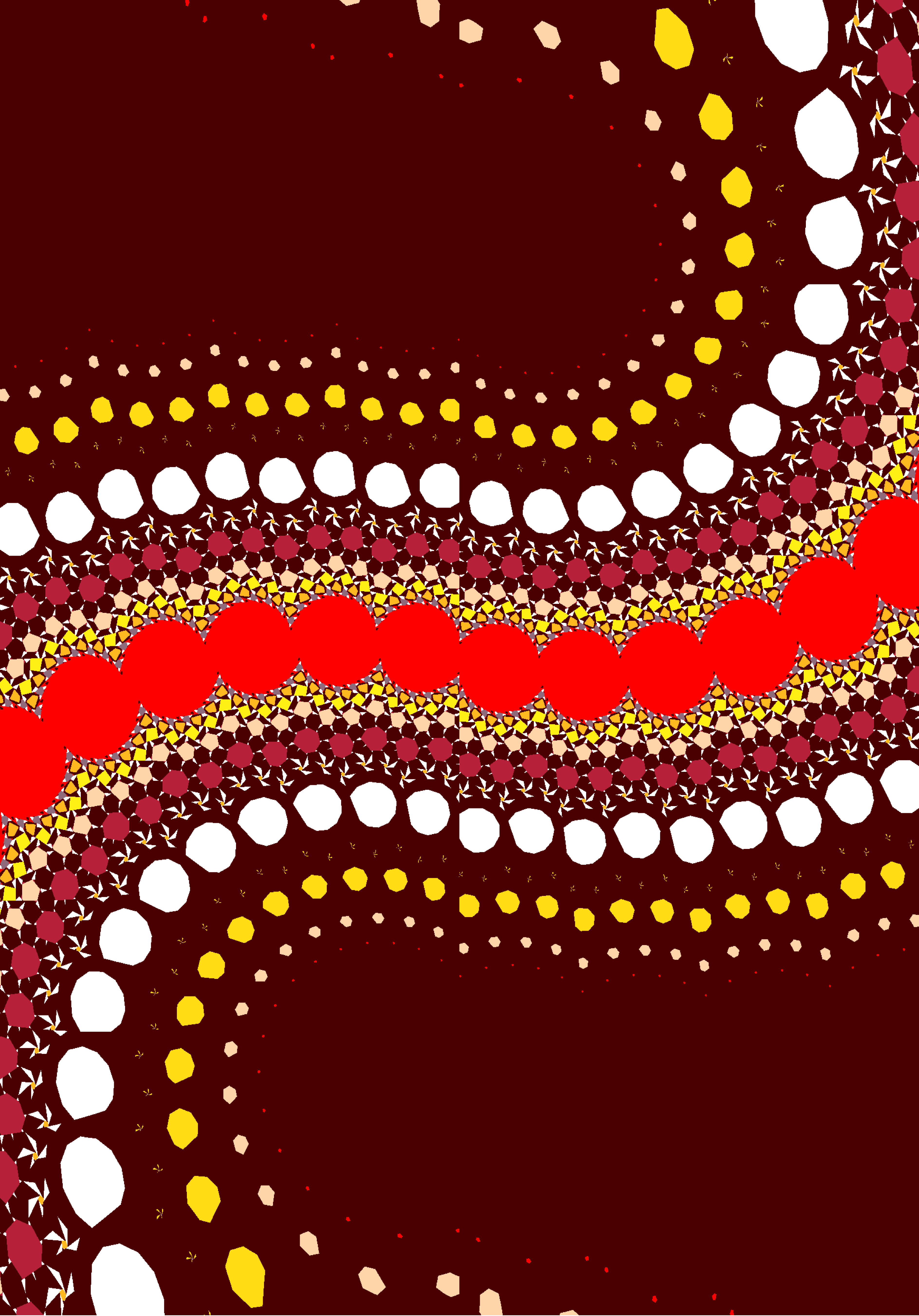}
		\end{center}
		\caption{Basins of attraction for an ellipse and line with colors based on Aboriginal Australian artwork. This image appears on the poster for \emph{MoCaO}.}\label{fig:ellipseart}
	\end{figure}
	Borwein et al. combined data from \emph{Cinderella} with parallelization techniques in order to visualize the global behaviour. An artistic rendering of the basins with colors chosen based on Aboriginal Australian artwork may be seen in Figure~\ref{fig:ellipseart}; this image appears on the poster for \emph{Mathematics of Computation and Optimization (MoCaO)}, an Australian Mathematical Society special interest group founded by Borwein and J\'er\^ome Droniou.
	
	Borwein et al. went on to show local convergence to feasible points in the case where the ellipse and line intersect non-tangentially, and they extended a best approximation result of Moursi and Bauschke \cite{BMshadow} in the setting of boundaries of convex sets. 
	
	In order to check the potential influence of sensitivity to compounding numerical error on their discoveries, Borwein et al. used Schwarzian reflection to compute approximate projections as an alternative to the numerical solution of a Lagrangian problem (see, for example, \cite{Needham}). This work inspired a 2017 follow up by Lindstrom, Sims, and Skerritt \cite{LSS}, analysing the performance of DR for finding intersections of smooth curves in $\mathbb{R}^2$ more generally and showing that local convergence extends to the more general case of smooth plane curves intersecting non-tangentially with reasonable limits on their curvature (in definition~\ref{def:superregular} we will introduce the notion of \emph{superregularity}). Dao and Tam \cite{DT} have since adapted Benoist's Lyapunov approach to beautifully illuminate the behaviour for more general curves, including showing global behaviour for many curve pairs.
	
	\begin{figure}[ht]
		\begin{center}
			\subfloat[Approximate solutions for an ODE]{\includegraphics[width=.3\textwidth]{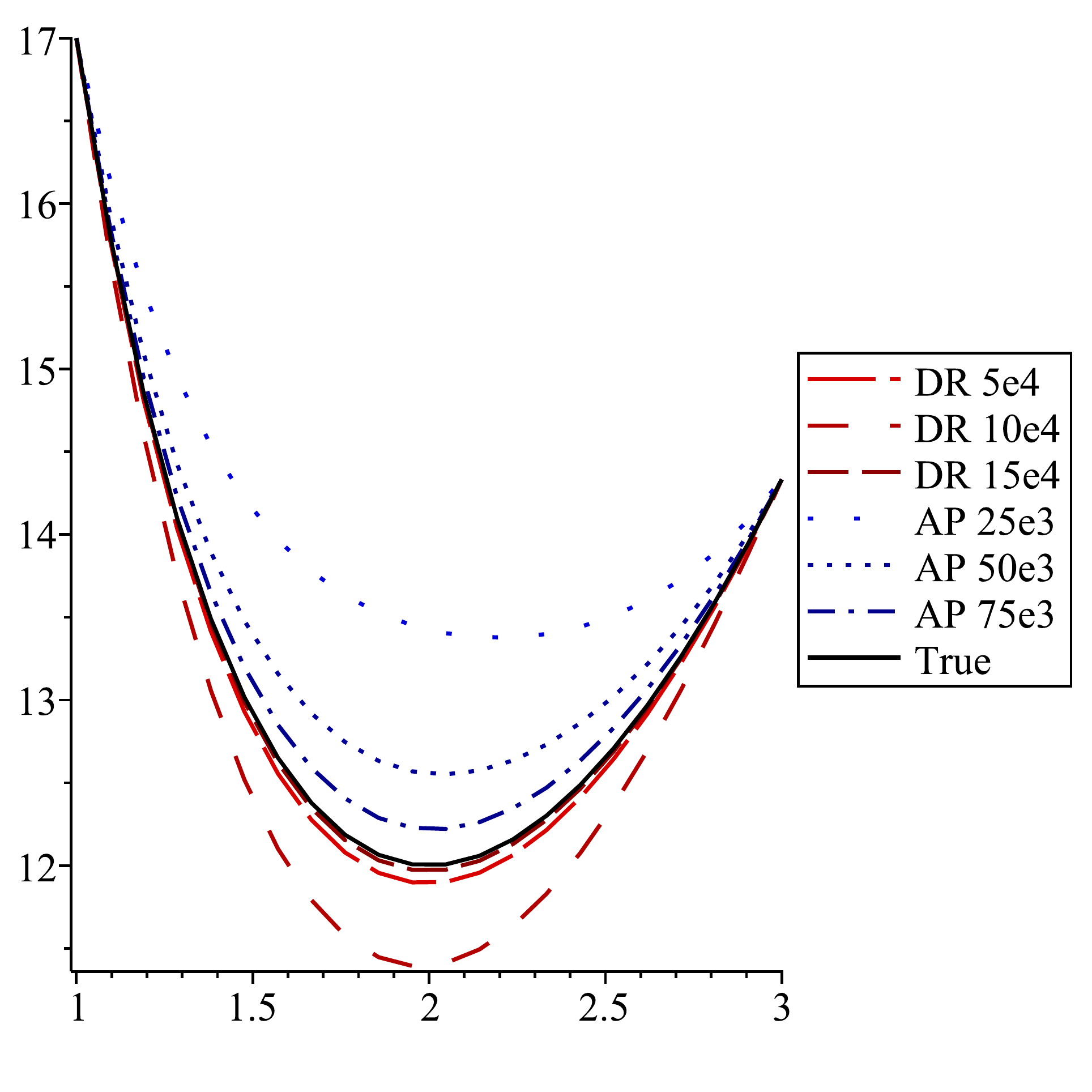}\label{fig:bookleft}}
			\subfloat[DR error plot]{\includegraphics[width=.3\textwidth]{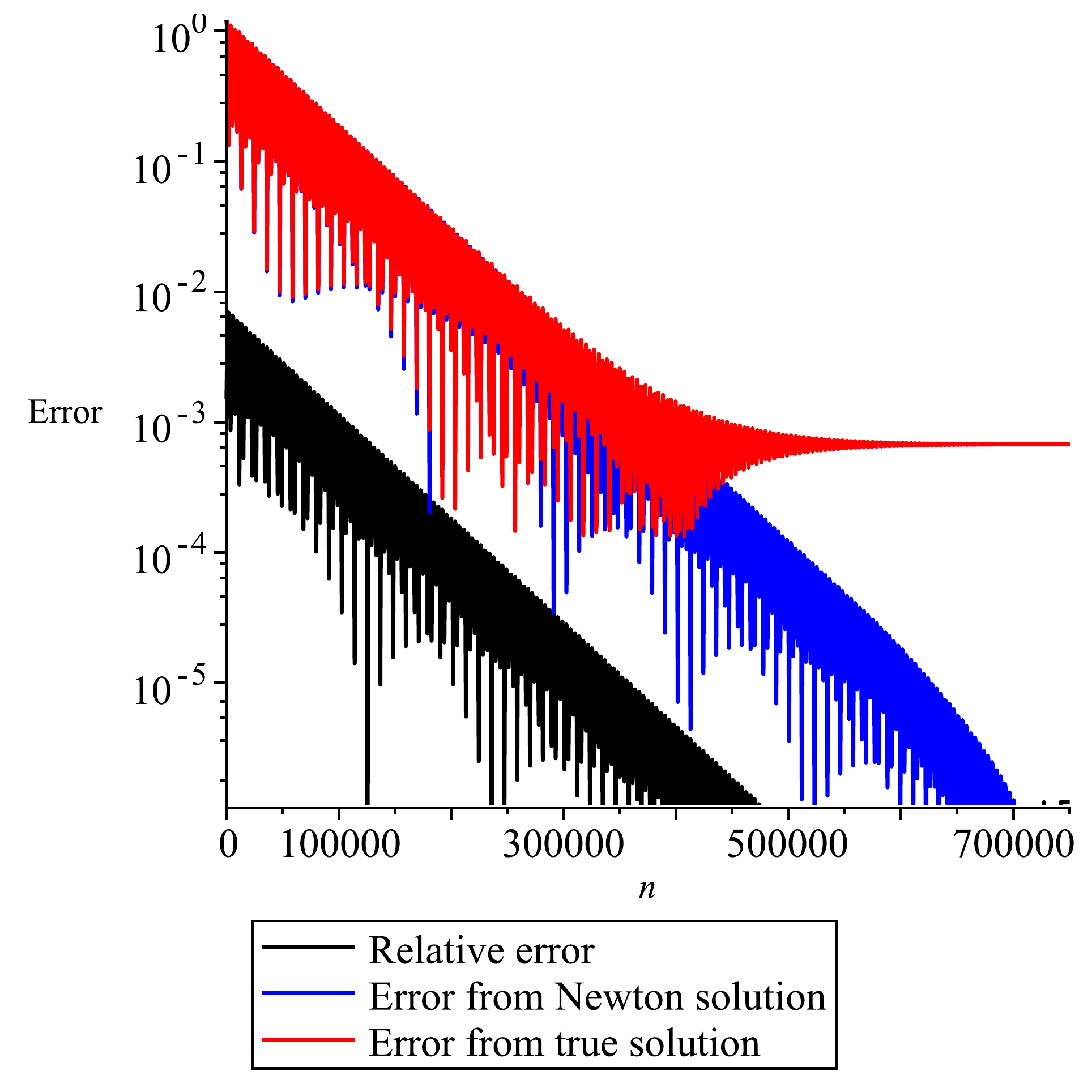}\label{fig:bookcenter}}
			\subfloat[AP error plot]{\includegraphics[width=.3\textwidth]{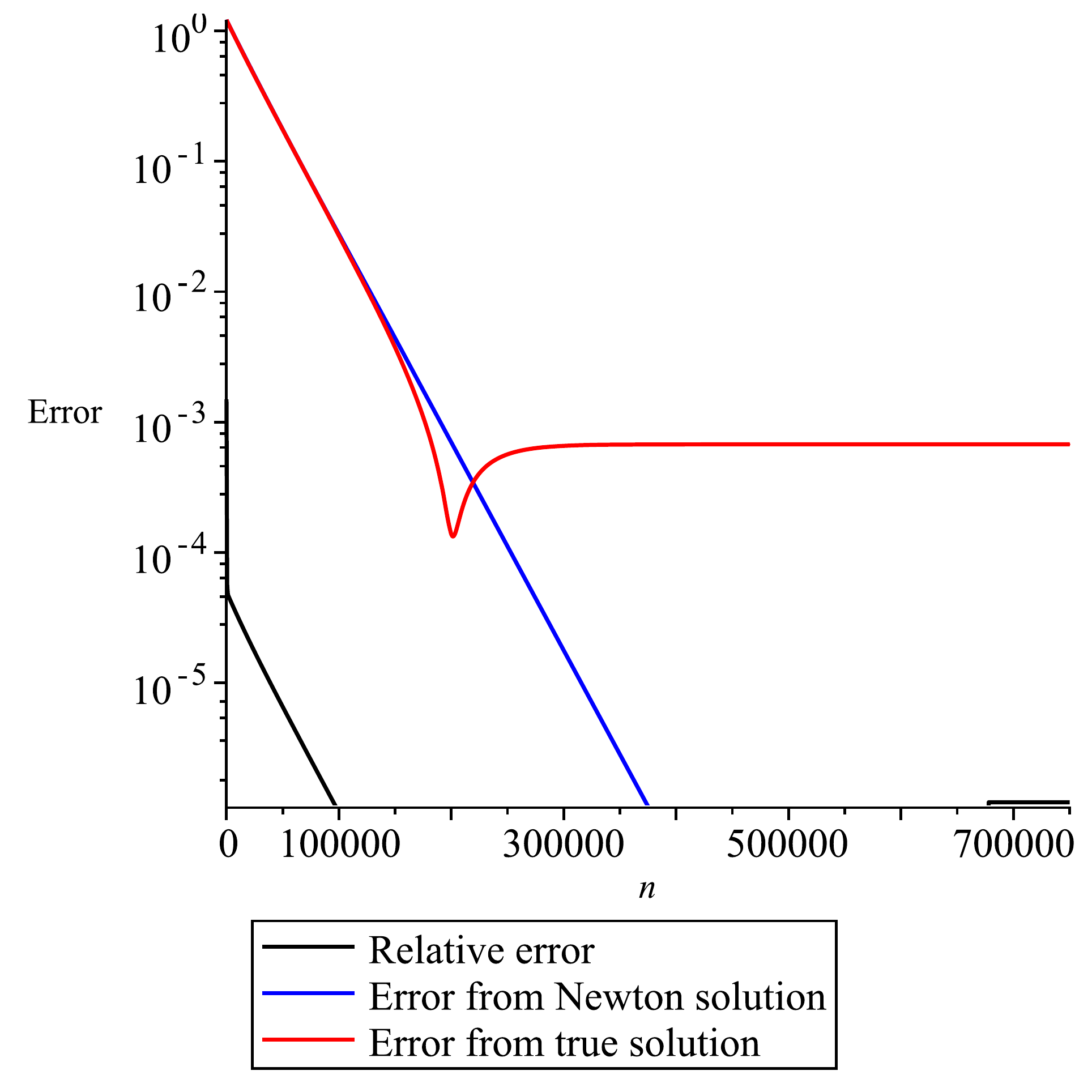}\label{fig:bookright}}
		\end{center}
		\caption{Comparison of DR and AP convergence behaviour}\label{fig:book}				
	\end{figure}
	
	Lamichhane, Lindstrom, and Sims used AP and DR to find numerical solutions for boundary value ODEs on closed intervals in $\mathbb{R}$ by reformulating the problem of $N$ node approximation as a feasibility problem of satisfying $N$ equations which define possibly discontinuous hypersurfaces \cite{LSS}. The approach is mostly experimental, and they compared the observed convergence to that explicitly visible in the $2$ set ellipse/line setting. They also compared the behaviour of DR and AP on each test problem, finding AP to generally perform faster.
	
	The above studies on hypersurfaces have uncovered a general trend which distinguishes AP from DR. Namely: AP is more prone to becoming trapped by extraneous fixed points but demonstrates monotonicity in convergence with an asymptotic direction of approach, while DR tends to escape from false solutions and its basins of convergence persistently feature spiralling trajectories which induce observed oscillations in plots of change and error. Some of this behaviour may be seen in Figure~\ref{fig:book} from \cite{LLS} which shows the behaviour, as measured for the agreement set shadow sequence $P_B x_n$, when seeking the solution to a $N$ set feasibility problem corresponding to the numerical solution of a boundary value problem. In Figure~\ref{fig:bookcenter} and Figure~\ref{fig:bookright}, relative error (change from iterate to iterate), error from numerical solution (obtained by applying Newton's method to the discretized problem), and error from the true solution (analytically obtained) are monotonic for AP but oscillate for DR. This monotonicity may be further observed in Figure~\ref{fig:bookleft} where approximate solutions to a boundary value problem---corresponding to various step intervals for DR and AP---may be seen along with the true solution; AP approaches the true solution from one side, while DR exhibits more exotic behaviour. The authors of \cite{LLS} hypothesize that the observed left-right-left wandering of $P_B x_n$ visible in Figure~\ref{fig:circleandlineleft} which results from the spiralling of $x_n$ is prototypical of the numerically observed oscillation in more complicated settings like Franklin's work on wavelets. 
	
	\subsection{Results on regularity, transversality, and rates of convergence}\label{transversal_section}
	
	Much of the convergence analysis in the nonconvex setting has focused on regularity assumptions. Throughout this section, $A$ and $B$ continue to be closed subsets of a finite-dimensional Euclidean space $X$.
	\begin{definition}[Regularity and transversality]The closed sets $\{C_i\}_{i \in I},\; I=\{1,\dots,m\}$ are said to be
		\begin{enumerate}
			\item $\kappa$\emph{-subtransversal} or $\kappa$\emph{-linearly regular} with regularity modulus $\kappa \in ]0,\infty[$ on $U\subset X$ if
			\begin{equation*}
			(\forall x \in U) \quad d_C(x) \leq \kappa \; \underset{i \in I}{\max}d_{C_i}(x),\quad \text{where}\quad C:=\bigcap_{i \in I}C_i ;
			\end{equation*}
			\item \emph{subtransversal} around $x \in X$ or \emph{linearly regular} at $x \in X$ if there exist $\delta$ and $\kappa$ greater than $0$ such that $\{C_i\}_{i \in I}$ is $\kappa$-linearly regular on $\mathbb{B}(x,\delta)$;
			\item \emph{boundedly linearly regular} if for every bounded set $U \subset X$ there exists $\kappa_U >0 $ such that $\{C_i\}_{i \in I}$ is $\kappa$-linearly regular on $U$.
			\item $U$-regular at $x \in X$ if $U$ is an affine subspace of $X$ with $x \in U$ and
			\begin{equation*}
			\sum_{i \in I}u_i = 0\quad \text{and}\quad u_i \in N_{C_i}(x)\cap (U-x) \; \implies \; (\forall i \in I) \; u_i = 0.
			\end{equation*}
			\item \emph{transversal} or \emph{strongly regular} at $x \in X$ if $\{C_i\}_{i \in I}$ is $U$-regular with $U=X$.
			\item \emph{affine-hull regular} at $x$ in the two set case $m=2$ when $L=\text{aff} (C_1 \cup C_2)$ if $N_A^L(x)\cap (-N_B^L(x)) = \{0\}$.
		\end{enumerate}
		See, for example, \cite{dao2016linear,kruger2006regularity,Phan,HL13}.
	\end{definition}
	More recently, the notion of \emph{intrinsic transversality} has been introduced which fills a theoretical gap between the regularity conditions of transversality and subtransversality \cite{drusvyatskiy2014alternating} (see also \cite{kruger2018intrinsic}).
	
	It may be readily seen that an ellipse and line which intersect non-tangentially are transversal at the point of intersection. Indeed the regularity framework locally describes many hypersurface feasibility problems. The notion of \emph{superregularity} for a single set $C$ may be thought of as a smoothness condition.
	\begin{definition}[superregularity]\label{def:superregular}A closed subset $A\subset X$ is $(\varepsilon,\delta)$-regular at $x$ if $\varepsilon \geq 0, \delta >0$, and
		\begin{equation*}
		\begin{rcases*}y,z \in A \cap \mathbb{B}_\delta(x)\\
		u \in N_A^{\text{prox}}(x) = \text{cone}(P_A^{-1}x-x) \end{rcases*} \implies \langle u, z-y \rangle \leq \varepsilon \|u\| \cdot \|z-y\|.
		\end{equation*}
		$C$ is said to be superregular at $x$ if for every $\varepsilon > 0$ there exists $\delta > 0$ such that $C$ is $(\epsilon,\delta)$-regular at $x$.
		See, for example, \cite{Phan}.
	\end{definition}
	It may be seen that in the case $X=\mathbb{R}^2$ and $A=\text{graph}f=\{(x_1,x_2) | f(x_1)=x_2 \}$ where $f:\mathbb{R} \rightarrow \mathbb{R}$, superregularity of $A$ at $(x,f(x))$ implies smoothness of $f$ at $x$.
	\begin{figure}[ht]
		\begin{center}
			\includegraphics[width=.5\textwidth]{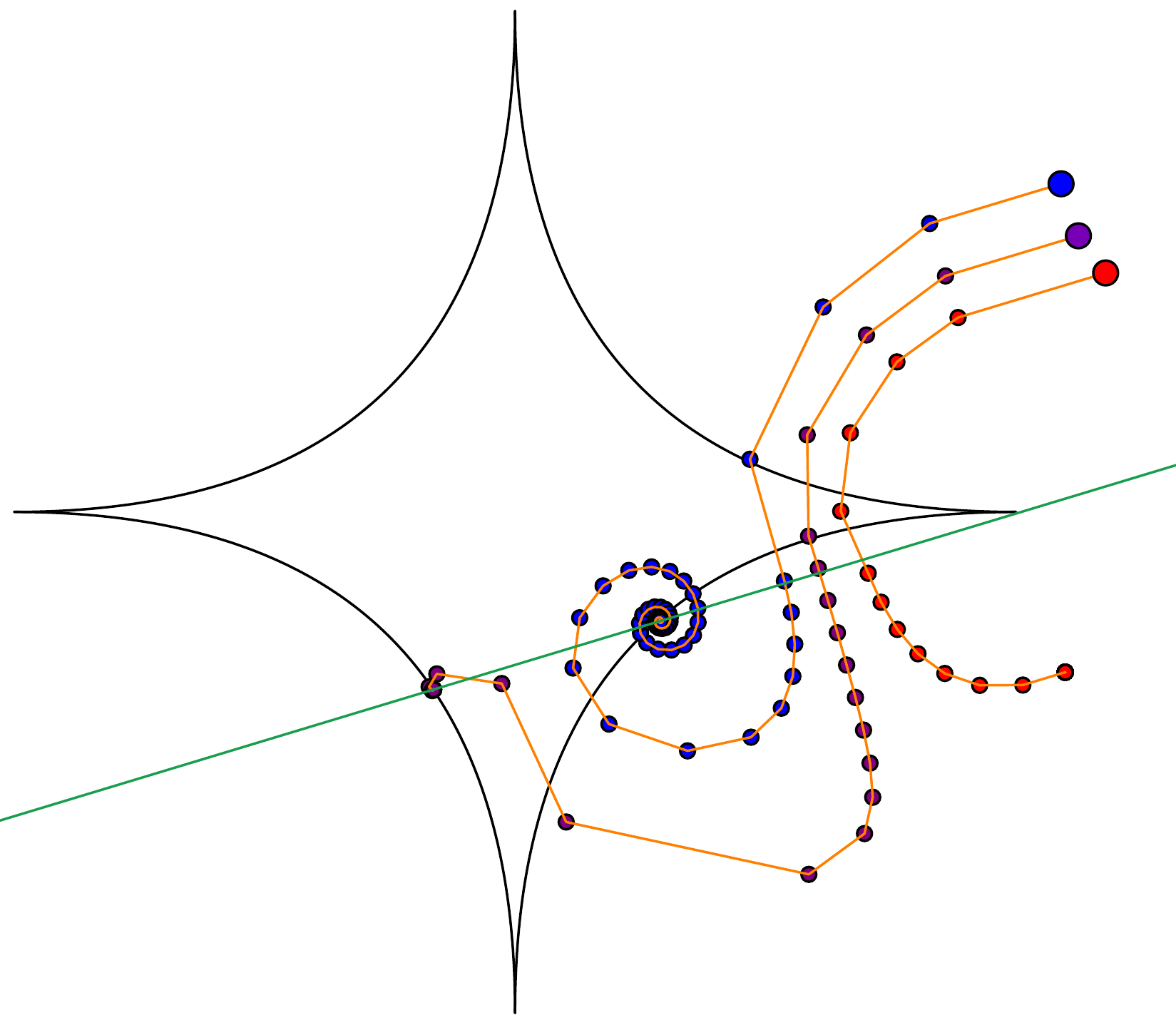}
		\end{center}
		\caption{DR convergence for a $1/2$-sphere and a line}\label{fig:psphere}
	\end{figure}
	
	Figure~\ref{fig:circleandlinecenter} shows how DR may behave when regularity conditions are not satisfied at the feasible point, while the rightmost sequence in Figure~\ref{fig:psphere} illustrates what may happen when two sets meet subtransversally but superregularity fails for one of them (the $p$-sphere). The other two sequences illustrate how the angle at which the sets meet at the feasible point determines the linear rate of convergence.
	
	As early as 2013, Lewis, Luke, and Malick analysed the local convergence for alternating and averaged nonconvex projection methods in the presence of regularity conditions \cite{LLM}. In the same year, Hesse and Luke undertook a theoretical study of DR in the presence of local regularity conditions in finite dimensions \cite{HL13}. They showed that when the sets involved are affine, strong regularity is necessary for linear convergence, in contradistinction with AP for which such conditions are sufficient but not necessary. They also established a number of linear convergence results for DR, the first of which is as follows.
	\begin{theorem}{Linear convergence of DR (Luke \& Hess, 2013 \cite[Theorem 3.16]{HL13})}
		Let the pair of closed sets $\{A,B\}$ be linearly regular at $x \in A\cap B$ on $\mathbb{B}_\delta(x)$ with regularity modulus $\kappa > 0$ for some $\delta > 0$. Suppose further that $B$ is a subspace and that $A$ is $(\varepsilon, \delta)$-regular at $x$ with respect to $A\cap B$. Assume that for some $c \in ]0,1[$ the following holds:
		\begin{equation*}
		\begin{rcases*}
		z \in A \cap \mathbb{B}_\delta (x), \quad u \in N_A(z)\cap\mathbb{B}_1(0)\\
		y \in B \cap \mathbb{B}_\delta(x), \quad v \in N_B(y) \cap \mathbb{B}_1(0)
		\end{rcases*} \implies \langle u, v \rangle \geq -c.
		\end{equation*}
		If $x_n \in \mathbb{B}_{\delta/2}(x)$ and $x_{n+1} \in T_{A,B}x_n$ then 
		\begin{equation*}
		d(x_{n+1},A\cap B) \leq \sqrt{1+2\varepsilon(1+\varepsilon)-\frac{1-c}{\kappa^2}} d(x_n,A\cap B).
		\end{equation*}
	\end{theorem} 
	Another of their results, \cite[Theorem 3.18]{HL13} has since been strengthened by Phan \cite[Theorem 4.3]{Phan} to the following.
	\begin{theorem}{Linear convergence of DR (Phan, 2016 \cite[Theorem 4.3]{Phan})}
		Let the closed sets $A,B$ be superregular at $x \in A \cap B$ and $\{A,B\}$ be strongly regular at $x$. Then if $x_0$ is sufficiently close to $x$, the sequence $x_{n+1}:=T_{A,B}x_n$ converges to a point $\overline{x} \in A \cap B$ with $R$-linear rate.
	\end{theorem}
	Phan provided additional information about the rate $R$ in \cite[Remark 4.5]{Phan}, and gave the following second main result on affine-hull regularity.
	\begin{theorem}{Linear convergence of DR (Phan, 2016 \cite[Theorem 4.7]{Phan})}Let $A,B$ be closed and $L:=\text{aff}(A\cup B)$. Further suppose $A,B$ are superregular at $x \in A \cap B$ and $\{A,B\}$ is affine-hull regular at $x$. Then, if the the shadow sequence $P_L x_0$ is sufficiently close to $x$, the DR sequence $x_{n+1}:=T_{A,B}x_n$ converges to a point $\overline{x} \in \Fix T_{A,B}$ with $R$-linear rate. Moreover,
		\begin{equation}
		P_A \overline{x} \equiv P_B \overline{x} = \overline{x} - (x_0-P_L x_0) \in A \cap B,
		\end{equation}
		and so $P_A \overline{x} \equiv P_B \overline{x}$ solves the feasibility problem.
	\end{theorem}
	Phan also provided a more detailed description of the region of convergence, and extended the analysis into the convex setting.
	
	In 2016 \cite{dao2016linear}, Dao and Phan went on to consider the more general framework of cyclic relaxed projection methods for the feasibility problem of $m$ sets $\{C_i\}_{i \in I}$, $I=\{1,\dots,m\}$ where the sequence is defined in terms of $l$ operators
	\begin{equation}\label{eqn:cyclicoperators}
	T_{nl+j}:=T_j \quad \text{and}\quad x_n := T_n x_{n-1} \quad \text{with}\quad J:=\{1,\dots,l\}.
	\end{equation}
	Where we have modified the notation to be consistent with \eqref{eq:defRS}, Dao and Phan considered the following cyclic generalized DR algorithm defined by \eqref{eqn:cyclicoperators} and the following. For every $j \in J$, let $\mu_j, \gamma_j \in [0,2[$, and $\lambda_j \in ]0,1[$, and $s_j,t_j \in I$ such that $s_j \neq t_j$ and 
	\begin{align*}
	I =& \{s_j | j \in J \} \cup \{t_j | j \in J\}, \\
	T_j :=& (1-\lambda_j)\Id +\lambda_j R_{C_{t_j}}^{\mu_j} R_{C_{s_j}}^{\gamma_j},
	\end{align*}
	where $R_{C_j}^{\mu_j}$ is defined as in \eqref{eq:defRS}. The convergence results are as follows.
	\begin{theorem}{Linear convergence of cyclic generalized DR (Dao \& Phan, 2016 \cite[Theorem 5.21]{dao2016linear})} Let $I:=\{1,\dots,m\}$ and $x \in \cap_{i\in I}C_i$. Suppose $\{C_i\}_{i \in I}$ is superregular at $x$ and linearly regular around $x$ and that $\{C_{s_j},C_{t_j}\}$ is strongly regular at $x$ for every $j \in J$. Then when started at a point $x_0$ sufficiently close to $x$, the cyclic generalized DR sequence generated by $(T_j)_{j \in J}$ converges $R$-linearly to a point $\overline{x}\in \cap_{i\in I}C_i$.
	\end{theorem}
	
	\begin{theorem}{Affine reduction for generalized DR sequences (Dao \& Phan, 2016 \cite[Theorem 5.25]{dao2016linear})}
		Let $A,B$ be closed, $x \in A \cap B$, and $L:=\text{aff}(A\cup B)$. Suppose $\{A,B\}$ is superregular and affine-hull regular at $x$. Let $(x_n)_{n \in \mathbb{N}}$ be defined by $x_{n+1}:=\left((1-\lambda)\Id + R_B^\mu P_A^\gamma \right)x_n$ where $\mu,\gamma \in [0,2[$ and $\lambda \in ]0,1[$. Then the following hold:
		\begin{enumerate}
			\item If $\gamma=\mu=0$, then, whenever $P_L x_0$ is sufficiently close to $x$, $(x_n)_{n \in \mathbb{N}}$ converges $R$-linearly to a point $\overline{x} \in \Fix T$ with $P_A \overline{x} = P_B \overline{x} \in A \cap B$.
			\item If either $\lambda > 0$ or $\mu > 0$, then, whenever $P_L x_0$ is sufficiently close to $x$, $(x_n)_{n \in \mathbb{N}}$ converges $R$-linearly to a point $\overline{x} \in A \cap B$.
		\end{enumerate}
	\end{theorem}
	
	\subsubsection{Other convergence results}
	
	Numerous other investigations of convergence for DR have also been undertaken. In 2014 Bauschke and Noll proved local convergence to a fixed point in the case where $A$ and $B$ are finite unions of convex sets \cite{BN14}. In 2016, Bauschke and Dao provided various sufficient conditions for finite convergence of the DR sequence \cite{BD17}. 
	
	\subsubsection{Further variants}	
	
	If one considers the spiralling behaviour characteristic of local convergence of DR, it is very natural to seek faster convergence by taking a step towards the center of the spiral. This intuition has given birth to the notion of \emph{circumcentering} the method \cite{circumcentering,behling2018linear}.		
	
	\subsubsection{Nonconvex minimization}\label{nonconvex_minimization}	
	
	In 2014 Patrinos, Stella and Bemporad introduced the so-called \emph{Douglas--Rachford envelope} whose stationary points correspond to solutions for the problem of minimizing a sum of two convex functions $f+g$ subject to linear constraints \cite{patrinos2014douglas}. 
	
	In 2015, motivated by properties of the Douglas--Rachford envelope, Li and Pong introduced the Douglas--Rachford merit function \cite{LP}:
	\begin{equation*}
	\mathfrak{D}_\eta (y,z,x) := f(y)+g(z)-\frac{1}{2\eta}\|y-z\|^2 + \frac{1}{\eta}\langle x-y,z-y\rangle.
	\end{equation*}
	Li and Pong analysed the limiting characteristics of $\mathfrak{D}_\eta (y_n,z_n,x_n)$ where $y_n,z_n,x_n$ are either as in \eqref{LP_notation} or are obtained from a modified variant where $x_0 \in X$ and
	\begin{align}\label{LiandPong}
	\begin{cases}
	y_{n+1} &= \frac{1}{1+\eta}\left(x_n+\eta P_A(x_n) \right)\\
	z_{n+1} &\in \underset{z \in B}{\argmin}\left \{\|2y_{n+1}-x_n-z_n \|  \right \}\\
	x_{n+1} &= x_n + (z_{n+1}-y_{n+1})
	\end{cases},
	\end{align}	
	which arises from applying \eqref{LP_notation} to the problem of minimizing $\frac{1}{2}d_A^2(x)$ subject to $x \in B$, where $A$ is convex but $B$ may not be. They showed the following.
	\begin{theorem}{Global subsequential convergence (Li \& Pong, 2015 \cite[Theorem 1]{LP})}\label{LPthm1}
		Let $g$ be proper and closed, $f$ have Lipschitz continuous gradient whose Lipschitz continuity modulus is bounded by $L$. Choose $\nu \in \mathbb{R}$ so that $f+\frac{\nu}{2}\|\cdot \|^2$ is convex. Suppose that $\eta$ is chosen so that $(1+\eta L)^2 + \frac{5\eta \nu}{2}-\frac{3}{2}<0$. Let $y_n,z_n,x_n$ be as in \ref{LP_notation}. Then $\{\mathfrak{D}_\eta \left(y_n,z_n,x_n \right) \}_{n\geq 1}$ is nonincreasing. Moreover, if a cluster point of $\left(y_n,z_n,x_n \right)$ exists, then 
		\begin{equation}\label{LiPong10}
		\underset{n\rightarrow \infty}{\lim} \|x_{n+1}-x_n \| = \underset{n\rightarrow \infty}{\lim} \|z_{n+1}-y_n \| = 0,
		\end{equation}
		and, for any cluster point $(\overline{y},\overline{z},\overline{x})$, we have $\overline{z}=\overline{y}$ and $0 \in \nabla f(\overline{z})+\partial g(\overline{z})$.
	\end{theorem}
	
	\begin{theorem}{Global convergence of the whole sequence (Li \& Pong, 2015 \cite[Theorem 2]{LP})}Let $f,g,l,L,x_n,y_n,z_n,\eta$ be as in theorem~\ref{LPthm1}. Additionally suppose $f,g$ are algebraic and that $\{(y_n,z_n,x_n) \}$ has a cluster point $(\overline{y},\overline{z},\overline{x})$. Then the sequence $\{(y_n,z_n,x_n) \}$ is convergent.		
	\end{theorem}
	
	\begin{theorem}{Convergence of DR splitting method for nonconvex feasibility problems involving two sets (Li \& Pong, 2015 \cite[Theorem 5]{LP})}
		Let $A$ be a nonempty, closed, convex set, and $B$ a nonempty closed set with either $A$ or $B$ compact. Suppose in addition that $0< \eta < \sqrt{3/2}-1$. Then the sequence $\{(y_n,z_n,x_n)\}$, where $y_n,z_n,x_n$ are as in \eqref{LiandPong}, is bounded. Moreover, any cluster point $(\overline{y},\overline{z},\overline{x})$ of the sequence satisfies $\overline{z}=\overline{y}$ and $\overline{z}$ is a stationary point of \eqref{LiandPong}. Additionally, \eqref{LiPong10} holds.
	\end{theorem}
	Li and Pong also provided detailed results on the convergence rates. Andreas Themelis and Panos Patrinos have since published a follow up article \cite{themelis2018douglas} in which they relax some of the restrictions on the step size $\eta$, as well as providing a discussion of the connections with ADMM.
	
	In 2017, Christian Grussler and Giselsson \cite{grussler2017local} analysed the specific case of minimizing $f+g$ with both forward-backward splitting and the Douglas--Rachford operator $T_{\partial f,\partial g}$ where $g$ is convex and
	\begin{equation*}
	f:M \mapsto k(\|M\|)+ \iota_{{\rm rank}(M) \leq r}(M)
	\end{equation*}
	is nonconvex where $k(\cdot)$ is increasing and convex, $\|\cdot \|$ is a unitarily invariant norm, and $\iota_{{\rm rank}(M) \leq r}$ is the indicator function for matrices that have at most rank $r$. They provided conditions under which ${\rm prox}_{f}$ and ${\rm prox}_{f^{**}}$ coincide, constructing a framework under which they showed local convergence when solutions to the convex problem of minimizing $f^{**}+ g$ coincide with solutions to the nonconvex problem of minimizing $f+g$.
	
	\section{Summary}\label{summary}
	
	The goal of this survey has been to illuminate the history, motivations and robustness of DR in each of the broad settings wherein it has been considered. Much more could be said, and certainly much more will be. As noted by Glowinski, Osher, and Yin in the preface of their new book on the subject, new applications of splitting methods are being introduced almost daily \cite{glowinski2017splitting}. 
	
	\subsection{Future avenues of inquiry}
	
	These directions include the continued analysis of the Arag\'{o}n Artacho-Campoy method in the convex setting, wavelet discovery in the nonconvex setting, nonconvex minimization through the framework of Li and Pong, and the analysis of convergence rates under general parameters in all of these. We choose to state here two problems in the nonconvex setting---both suggested by Veit Elser---which have received little attention despite their particularly intriguing nature.
	
	\subsubsection{Continuous time variant}
	
	\begin{figure}[ht]	
		\begin{center}
			\includegraphics[height=.25\textheight]{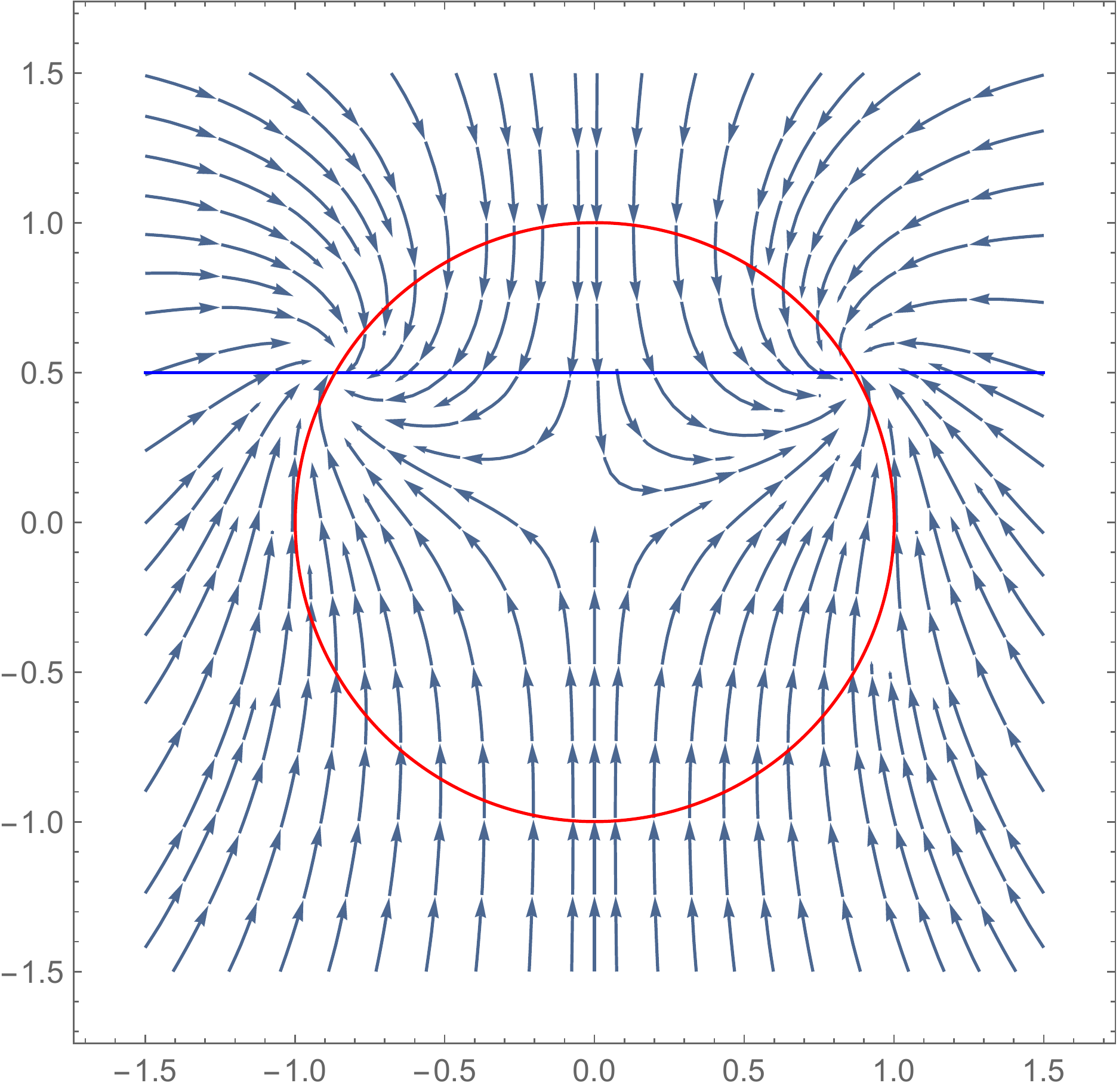}\quad \includegraphics[height=.25\textheight]{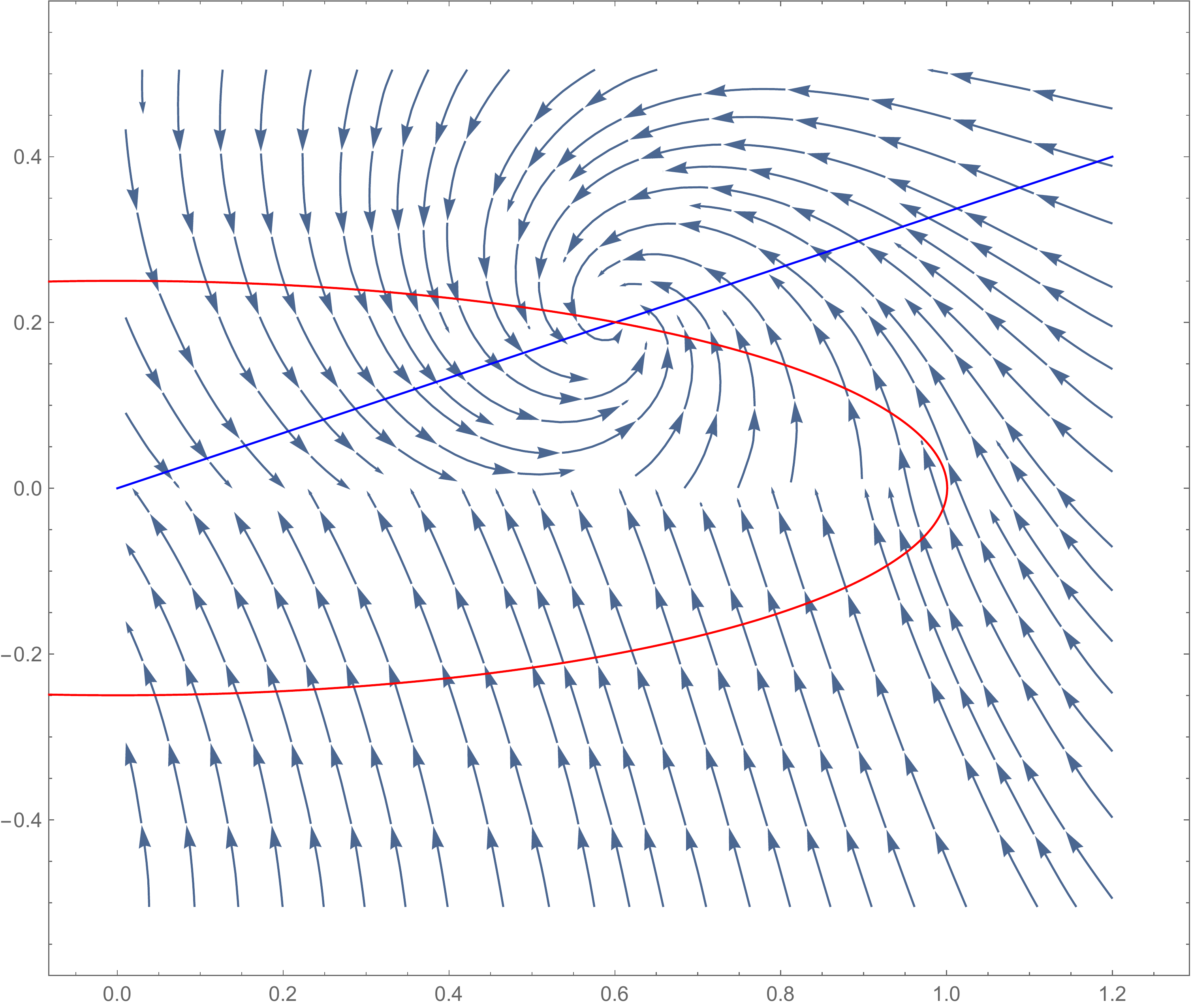}
		\end{center}
		\caption{The flowfield \eqref{DR_DE} with a circle/line (left) and ellipse/line (right). Images courtesy of Veit Elser.}\label{fig:circleflow}
	\end{figure}
	
	For the case of a circle and line, Borwein and Sims \cite{BS} considered the ``continuous time'' version of the algorithm whose flow field is at left in Figure~\ref{fig:circleflow} and corresponds to the solution of the differential equation
	\begin{equation*}\label{DR_DE}
	\frac{dx}{dt}=T(x) \quad \text{when} \; \lambda \rightarrow 0^+ .
	\end{equation*} 
	Veit Elser has suggested analysing the continuous time variant in the more general setting of ellipses and plane curves. Elser provided flow field images for a curve and integer lattice in \cite{elser2017matrix}, and he has generously furnished the images in Figure~\ref{fig:circleflow}.
	
	\subsubsection{Complexity Theory}\label{complexity_theory}
	
	Elser hypothesizes that, for Latin square problems, higher dimensionality is associated with greater robustness for the algorithm. The idea is that as the complexity of the problem grows, the singular regions---of chaotic or periodic behaviour---account for a smaller share of the total space. For most starting points, then, the iterates tends to explore the space without becoming stuck, as in Figure~\ref{fig:DRvsAPleft}, until eventually they fall into the basin of attraction for a feasibility point. Evidence abounds, as in \cite{elser2017benchmark,elser2017matrix,elser2018complexity}. Can the behaviour of DR and similar methods under complexity be rigorously catalogued through experimental analysis?
	
	\subsection{Conclusion}
	
	The role of DR in the convex setting is both well-known and celebrated. More novel and striking is its success in the nonconvex setting. Jonathan Borwein described DR as an ``out-of-the-box solver,'' whose robustness for a given nonconvex problem cannot be simply explained by its having been originally designed with that specific problem in mind. While the exact formulation for an embedding of a problem in $\mathbb{R}^n$---for example, the stochastic representation of a sudoku puzzle or the number of gadgets used in \cite{AC}---may affect performance, DR fundamentally requires very little: if one can compute the projections, one can use the solver. Perhaps this is why its performance consistently surprises those who study or use it. One thing is certain: the complexity of the behaviour is astounding, and much of the space remains to be explored.
	
	\subsection*{Acknowledgements}
	
	We are particularly grateful to Heinz Bauschke and Veit Elser, whose expertise on DR has been instrumental in our reconstruction of its history. We also extend our thanks to Pontus Giselsson, whose careful reading and helpful suggestions helped to improve the manuscript.

	\bibliographystyle{plain}
	\bibliography{bibliography_updated}
	
	\section{Appendix: ADMM and Douglas--Rachford}\label{appendix:ADMM}
	
	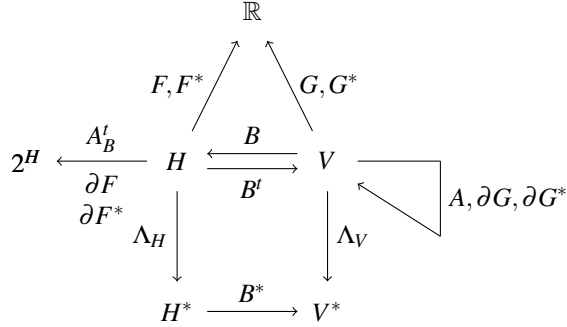
\begin{figure}
		\begin{center}
			\begin{tikzpicture}[scale=2.0]
			\node [] at (0,0) {$H^*$};
			\node [] at (0,1) {$H$};
			\node [] at (-1,1) {$2^H$};
			
			\node [] at (1,0) {$V^*$};
			\node [] at (1,1) {$V$};
			\node [] at (-1,1) {$2^H$};
			\node [] at (0.5,2) {$\mathbb{R}$};
			
			\draw [<-] (0,.2) -- (0,.8);
			\node [left] at (0,.5) {$\Lambda_H$};
			\draw [<-] (1,.2) -- (1,.8);
			\node [right] at (1,.5) {$\Lambda_V$};
			\draw [->] (.2,0) -- (0.8,0);
			\node [above] at (.5,0) {$B^*$};
			\draw [<-] (.2,1.05) -- (0.8,1.05);
			\node [above] at (.5,1.05) {$B$};
			\draw [->] (.2,.95) -- (0.8,.95);
			\node [below] at (.5,.95) {$B^t$};
			\draw [->] (-.2,1) -- (-0.8,1);
			\node [above] at (-.5,1) {$A_B^t$};
			\node [below] at (-.5,1) {$\begin{array}{c}\partial F\\
				\partial F^* \end{array}$};
			
			\draw [->] (0.1,1.2) -- (0.4,1.8);
			\node [left] at (.25,1.5) {$F,F^*$};
			\draw [->] (0.9,1.2) -- (0.6,1.8);
			\node [right] at (.75,1.5) {$G,G^*$};
			
			\draw [] (1.2,1) -- (1.75,1);
			\draw [] (1.75,1) -- (1.75,.5);
			\draw [->] (1.75,0.5) -- (1.2,.85);
			\node [right] at (1.75,.75) {$A, \partial G,\partial G^*$};
			\end{tikzpicture}
		\end{center}
		\caption{Function diagram for Gabay's exposition.}\label{fig:Gabay}
	\end{figure}
	
	Throughout this section, the function diagram in Figure~\ref{fig:Gabay} is a useful reference. In particular, it should be noted that Gabay defined the conjugates $F^*:H\rightarrow \mathbb{R}$ and $G^*:V \rightarrow \mathbb{R}$ on the primal spaces.
	
	In 1983 \cite{Gabay}, Daniel Gabay considered the application of \eqref{LM} with $\B := \partial F^* = (\partial F)^{-1}$ for $F:H \rightarrow ]0,\infty]$ a proper convex lsc function, $\A := A_B^t :H\rightarrow 2^H$ by $$A_B^t(\mu) = \{q \in H | \exists v \in V\; \text{such that}\;  q=-Bv,\; -B^t\mu \in A(v)\},$$
	for a maximally monotone operator $A$, and $B: V \rightarrow H$ is a continuous linear operator with adjoint $B^*:H^* \rightarrow V^*$
	\begin{align*}
	\text{where}\;&\begin{cases}
	\langle \Lambda_V u, v\rangle_{V^* \times V} = \langle u,v \rangle_V \;\forall\; u,v \in V \;\;\text{with}\;\; \Lambda_V u \in V^*\\
	\langle \Lambda_H p, q\rangle_{H^* \times H} = \langle p,q \rangle_H\; \forall\; p,q \in H \;\;\text{with}\;\; \Lambda_H p \in H^*
	\end{cases};\\
	\text{and}\quad & B^t: H \rightarrow V \;\text{by}\; B^t:= \Lambda_V^{-1} \circ B^* \circ \Lambda_H .
	\end{align*}
	The motivating variational inequality problem is to find $u \in V$ such that
	\begin{equation}\label{Gabay_var_ineq}
	\exists w \in A(u)\;\;\text{where}\;\;(\forall v \in V) \;\; \langle w, v-u\rangle_V+F(Bv)-F(Bu) \geq 0.
	\end{equation}
	When $A=\partial G$ for $G:V \rightarrow ]-\infty,\infty]$ a convex, proper, lsc function, the variational inequality \eqref{Gabay_var_ineq} is just
	\begin{equation}\label{Gabay_optim}
	\mathbf{p}:=\underset{v \in V}{\inf} \{F(Bv)+G(v)\}.
	\end{equation}	
	When $A$ is coercive or $B^t B$ is an isomorphism of $V$, then $$J_{A_B^t}^\lambda (y) = y+\lambda B(A+\lambda B^t B)^{-1}(-B^t y).$$ Gabay showed that \eqref{LM} then becomes:
	\begin{align}
	\textbf{Step 0}&\quad \text{choose}\;\;\omega_0 \;\; \text{to be an approximate solution of the problem:}\nonumber \\
	&\quad \text{Find}\;\omega\;\text{such that}\;0 \in (A_B^t+\partial F^*)(\omega)\nonumber \\
	\textbf{Step 1}&\quad \text{choose}\;\; x_0,p_0 \;\; \text{such that}\quad p_0 \in \partial F^*(\omega_0),\; t_0 = \omega_0+\lambda p_0  \nonumber \\
	&\quad (\text{This ensures}\quad \omega_0 = J_{\partial F^*}^r(x_0))\nonumber \\
	\textbf{Step 2}&\quad
	\begin{cases}
	u_{n+1} & := (A+\lambda B^tB)^{-1}(\lambda B^t p_n - B^t \omega_n)\\
	p_{n+1} &:= (\partial F+\lambda \Id)^{-1} (\omega_n + \lambda B u_{n+1})\\
	\omega_{n+1} &:= \omega_n + \lambda(B u_{n+1}-p_{n+1})\\
	x_{n+1} &:= \omega_n + \lambda Bu_{n+1} 
	\end{cases}\label{ALG2}
	\end{align}
	In this case, $(\omega_n)_{n \in \mathbb{N}}$ is the sequence of multipliers, and $\omega_n := J_{\partial F^*}^\lambda (x_n)$ is the shadow sequence iterate corresponding to the $n$th iterate of the Douglas--Rachford sequence $(x_n)_{n \in \mathbb{N}}$. In terms of Figure~\ref{fig:APandDRright}, if we take  $\lambda=1$, $B=\Id$, $F^* = N_A$, and $A_B^t = N_B$, then, in \eqref{ALG2}, $x_n = x_n$, $\omega_n = P_A x_n$, $p_n = (x_n - P_A x_n)$, and $u_{n+1} = (P_B R_A x_n - R_A x_n)$.
	
	Gabay rewrites \eqref{ALG2} as in terms of the sequences $u_n,p_n,\omega_n$:
	\begin{align}
	\textbf{Step 0}&\quad \text{Find}\; u_{n+1}\in V \text{satisfying the variational inequality:} \;\; \exists w_{n+1} \in A(u_{n+1})\; \nonumber \\
	&\quad \text{such that}\; (\forall v \in V) \langle w_{n+1},v\rangle_V + \langle \omega_n - \lambda p_n + \lambda Bu_{n+1}, Bv\rangle_H = 0 \nonumber \\
	\textbf{Step 1} &\quad \text{Find}\; p_{n+1}\;\text{which solves the minimization problem:} \nonumber\\
	&\quad F(p_{n+1})-F(q)-\langle \omega_n,p_{n+1}-q \rangle_H + \frac{\lambda}{2}\|Bu_{n+1}-p_{n+1}\|_H^2 - \frac{\lambda}{2}\|Bu_{n+1}-q \|_H^2 \leq 0 \nonumber\\
	\textbf{Step 2}&\quad \text{Update multiplier by}\quad \omega_{n+1} \leftarrow \omega_n + \lambda(Bu_{n+1}-p_{n+1}).\nonumber
	\end{align}
	Gabay highlights that this is a variant of Uzawa's algorithm \cite{uzawa} for the augmented Lagrangian 
	\begin{equation*}
	\mathfrak{L}_r (v,q,u) = F(q)+G(v)+\langle \mu, Bv-q \rangle_H + \frac{\lambda}{2}\|Bv-q \|_H^2
	\end{equation*}
	for solving the optimization problem \eqref{Gabay_optim}. When $A=\partial G$, under qualification conditions, $A_B^t = \partial (G^* \circ (-B^t))$ and so
	\begin{equation}\label{Gabay_dual}
	\mathbf{d}:=\underset{\mu \in H}{\inf} \{G^*(-B^t\mu)+F^*(\mu)\}
	\end{equation}	
	is the dual value associated with the primal value \eqref{Gabay_optim}. See, for example, \cite[Theorem 3.3.5]{BL}. Thus the Lagrangian method of Uzawa applied to finding $\mathbf{p}$ \eqref{Gabay_optim} is equivalent to DR applied to finding $\mathbf{d}$ \eqref{Gabay_dual}.

	
\end{document}